\documentclass[11pt,twoside]{article}  
\usepackage[utf8]{inputenc}
\usepackage[T1]{fontenc}
\usepackage{mathtools,amssymb, amsthm}            
\usepackage{amstext, amsfonts, a4}
\usepackage{color}
\usepackage{enumerate}
\usepackage{csquotes}
\usepackage{graphicx}
\usepackage{todonotes}

\newcommand{\R}{\mathbf{R}}
\renewcommand{\S}{\mathbf{S}}

\newcommand{\N}{\mathbf{N}}
\newcommand{\E}{\mathbb{E}}

\newcommand{\eps}{\epsilon}
\newcommand{\Om}{\Omega}
\newcommand{\om}{\omega}
\newcommand{\Id}{\operatorname{Id}}
\newcommand{\oo}{\infty}
\newcommand{\lt}{\left}
\newcommand{\rt}{\right}

\theoremstyle{plain}
\newtheorem{thm}{Theorem}[subsection]
\newtheorem{lem}[thm]{Lemma}
\newtheorem{prop}[thm]{Proposition}

\theoremstyle{definition}
\newtheorem{definition}{Definition}[subsection]
\theoremstyle{remark}
\newtheorem{rem}{Remark}[subsection]
\newtheorem{ex}{Example}[subsection]

\title{Stochastic homogenization of the Landau-Lifshitz-Gilbert equation}
\author{Fran\c{c}ois Alouges\footnote{CMAP, Ecole polytechnique, CNRS, Institut Polytechnique de Paris, 91128 Palaiseau Cedex, France, {\tt{francois.alouges@polytechnique.edu}}},
Anne de Bouard\footnote{CMAP, Ecole polytechnique, CNRS, Institut Polytechnique de Paris,
91128 Palaiseau Cedex, France, {\tt{anne.debouard@polytechnique.edu}}}, 
Beno\^\i t Merlet\footnote{LPP, CNRS UMR 8524, Universit\'e de Lille, F-59655 Villeneuve d'Ascq Cedex, France and Team RAPSODI, Inria Lille - Nord Europe, 40 av. Halley, F-59650 Villeneuve d'Ascq, France, {\tt{benoit.merlet@univ-lille.fr}}},
L\'ea Nicolas\footnote{CMAP, Ecole polytechnique, CNRS, Institut Polytechnique de Paris, 91128 Palaiseau Cedex, France, 
{\tt{lea.nicolas@polytechnique.edu}}}}
\date{}

\begin{document}

\maketitle

\begin{abstract}
Following the ideas of V. V. Zhikov and A. L. Piatnitski~\cite{zhikov2006homogenization}, 
and more precisely the stochastic two-scale convergence, this paper establishes a 
homogenization theorem in a stochastic setting for two nonlinear equations~: the 
equation of harmonic maps into the sphere and the Landau-Lifschitz equation. These equations have strong nonlinear features, in particular,  in general their solutions are not unique.
\end{abstract}

\section{Introduction}

Magnetic materials are nowadays at the heart of numerous electric devices
(engines, air conditioning, transportation, etc.). Very often, the
efficiency of the device is directly linked to the magnetic properties
of the magnets used \cite{inproceedings}. Best magnets are the rare-earth magnets 
(e.g. Samarium-Cobalt or Neodymium magnets) which have
been developed since the 1980's with no equivalent yet. Nevertheless, 
due to the uneven distribution of rare-earth ores~\cite{fullerton1999hard}, 
the magnet manufacturers aim at developing new types of magnets that do 
not require rare-earth. Among the best promising candidates are the so-called 
\emph{spring magnets} which are made of hard and soft magnets intimately 
mixed at the nanoscale~\cite{gutfleisch2011magnetic}. 

Mathematically speaking, the problem of studying such materials is challenging
because usual models are highly non-linear and the material dependence of 
the parameters varies at a very small scale. It is therefore a problem of homogenization.
A complete static model, including all relevant classical physical terms, has been 
derived in~\cite{alouges2015homogenization} in terms of $\Gamma-$convergence 
of the minimization problems related to the underlying so-called Brown energy of 
such materials. Nevertheless, to the knowledge of the authors, there is no such
study to characterize the dynamic problem, or, in other words, the evolution of the 
magnetization inside the composite ferromagnet. This is the first purpose of this paper,
while the second is to consider that the different compounds are randomly distributed at the
microscopic scale inside the macroscopic composite, whereas a periodic distribution was considered
in \cite{alouges2015homogenization}.

To go more precisely into the details, the magnetization inside a homogeneous 
ferromagnetic materials obeys the Landau-Lifschitz equation (see for instance~\cite{hubert2008magnetic}), 
a non-linear PDE, which, when only the exchange 
interactions between magnetic spins are taken into account reads as
\begin{equation}
    \frac{\partial u}{\partial t} = u \times \operatorname{div}\left(a\nabla u\right) - 
    \lambda u\times \left(u\times\operatorname{div}\left(a\nabla u\right)\right)\,.
\label{LL1}
\end{equation}

Here, $u(t,x)$ is a unit vector in $\R^3$ that denotes the magnetization at time  $\;$
$t\geq 0$ and at $x\in \mathcal{D}$, where $\mathcal{D}$ is the bounded domain of $\R^3$ occupied by the material,
and the symbol $\times$ stands for the cross product in $\R^3$.
The so-called \emph{exchange parameter} $a$ is a $3\times3$ matrix, which depends 
on the considered material(s). We also assume the different materials to be strongly coupled which amounts to say that the direction of the magnetization $u$ does not jump at the interface between two materials.

Spring magnets being composed of several different materials, the coefficient $a$ is 
likely to depend on the space variable. Assuming furthermore that the materials are 
randomly distributed, and on a small scale $\eps$, we are led to consider the Landau-Lifschitz equation with random coefficients
\begin{equation}
   \left\{ 
   \begin{aligned}
     &\frac{\partial u^{\eps}}{\partial t} =  u^{\eps}\times 
     \operatorname{div}(a(\frac{x}{\eps},\om)\nabla u^{\eps})  
     - \lambda u^{\eps} \times \left(u^{\eps}\times 
     (\operatorname{div}(a(\frac{x}{\eps},\om)\nabla u^{\eps})\right)\, , \\
     & u^{\eps}(x,0) = u_0(x)\,, \\
     & |u^{\eps}(x,t)| = 1 \text{ for a.e. } (x,t)\in Q_T\,, \\
     & a(\frac{x}{\eps},\om)\nabla u^{\eps}\cdot n = 0 \text{ on }
     \partial \mathcal{D}\times (0,T)\,\end{aligned}\right.\label{stocLL}
     \end{equation}
with $T>0$, $Q_T=\mathcal{D}\times (0,T)$, and $u_0 \in H^1(\mathcal{D},\R^3)$ 
satisfying $|u_0(x)|=1$ a.e. in $\mathcal{D}$. The exchange parameter $a$ depends 
on $x$ at scale $\eps$ and on the random parameter $\om$.\\
The problem we wish to solve is therefore the stochastic homogenization of the 
Landau-Lifschitz equation, or in other words, passing to the limit in~\eqref{stocLL} 
as $\eps$ goes to 0. Notice that the problem possesses several difficulties that 
make it not obvious: global solutions of~\eqref{LL1} are only known to exist weakly 
and are not unique~\cite{alouges1992global, visintin1985landau}, the constraint 
$|u(x,t)|=1$ is not convex and the equation is highly non linear.\\
In order to proceed, we apply the stochastic two-scale convergence method. 
Originally defined in a periodic deterministic framework for the first time by G. Nguetseng~\cite{nguetseng1989general}, 
the theory was further developed by G. Allaire~\cite{allaire1992homogenization} and is by now currently used. Stochastic 
homogenization of PDEs dates back to~\cite{zbMATH03787778} for linear 
equations and~\cite{dal1985nonlinear} for nonlinear problems. Furthermore, a stochastic 
generalization of two-scale convergence (in the mean) has been first proposed in~\cite{bourgeat1994stochastic}, 
but turns out to be inadequate for our purposes. 
Instead, we use in this paper a theory developed in~\cite{zhikov2006homogenization}. 
This latter version allows us to realize the proposed program: we prove that, up to extraction, $u^{\eps}$ converges (weakly in $H^1(Q_T)$) to a (weak) solution of~\eqref{LL1} where the effective 
exchange matrix $a=a^{\text{eff}}$ is fully identified. \\
    
The paper is organized as follows. In Section 2, we recall the probability setting 
and introduce the stochastic two-scale convergence. For the sake of completeness, we 
present a complete theory, simpler than the one given in~\cite{zhikov2006homogenization},
but sufficient for our purpose. Most of the arguments are nevertheless borrowed from~\cite{zhikov2006homogenization}. 
Section 3 is devoted to the applications. The classical
diffusion equation is first quickly treated and we turn to the stochastic homogenization 
of the equation of harmonic maps into the sphere. Eventually, the Landau-Lifschitz equation 
is considered.
    
\section{Stochastic two-scale convergence}
A stochastic generalization 
of the two-scale convergence method, was proposed in~\cite{zhikov2006homogenization} 
in a very general context. In order to have the present paper self-contained, 
we restrict the approach of~\cite{zhikov2006homogenization} to a framework 
that is sufficient for our needs and we present a complete setting. 
In our opinion, this method did not have the resonance that it deserves and we hope that the  reader will find 
here a comprehensive introduction to it.

\subsection{The stochastic framework}
Let $(\Om,\mathcal{F},\mathbb{P})$  be a standard probability space with 
$\mathcal{F}$ a complete $\sigma$-algebra and let us consider a $d$-dimensional random field $a(x,\om)\in \R^{d\times d}$, defined for $x$ in $\R^d$ and $\om$ in $\Om$. We consider a group action of $\R^d$ on the set $\Om$, for $x\in\R^d$, we note $T_x$ the action on $\Om$ and we call it the translation of vector $x$.  We assume that $a$ and $T$ satisfy the following classical 
assumptions:
\begin{enumerate}[(H1)]
\item Compatibility: the mapping $(x,\om)\mapsto T_x\om$ is measurable from $\R^d\times\Om$ into $\Om$. Moreover, for every $x$ in $\R^d$ and $\om$ in $\Om$, $a(\cdot,T_x\om)=a(x+\cdot,\om)$.  
    \item Stationarity:  for every $x$ in $\R^d$, for every $k$ in $\N$, for every 
    Borelian $\mathcal{B}$ of $(\R^{d\times d})^k$, for every $y_1,\ldots,y_k$ in $\R^d$,
    \begin{equation*}
        \mathbb{P} \left\{(a(y_1,T_x\om),\ldots,a(y_k,T_x\om))\in \mathcal{B}\right\} = \mathbb{P} 
        \left\{(a(y_1,\om),\ldots,a(y_k,\om))\in \mathcal{B}\right\}.
    \end{equation*}
    \item Ergodicity: the only measurable sets that are translation invariant (that is, $A\in\mathcal{F}$ such that (up to a null subset) $T_xA=A$ for every $x\in \R^d$) have null or full measure.
        \item The matrix $a$ is symmetric, uniformly bounded and uniformly elliptic:
    \begin{equation*}
       \exists c_1,c_2>0;\;\forall \om\in \Om,\;\forall \xi\in \R^d, \;
       \forall x\in \R^d,\; c_1|\xi|^2 \leq \xi\cdot a(x,\om)\xi\leq c_2|\xi|^2. 
    \end{equation*}
    \item The matrix $a$ is \emph{stochastically continuous}: for every $x$ in $\R^d$,
        \begin{equation*}
           \forall \eps>0, \; \lim_{y\rightarrow x} \mathbb{P}(\|a(y)-a(x)\|_{\R^{d\times d}}\geq \eps) =0 \, .
        \end{equation*}
\end{enumerate}
As we shall see, the assumption (H5) allows us to restrict ourselves to the case where 
$\Om$ is a compact metric space, which is the assumption used in~\cite{zhikov2006homogenization}. 
In fact, all the results given in this paper are 
applicable if (H5) is replaced by assumptions (H5'-a)--(H5'-c) below.
\begin{enumerate}
  \item[(H5'-a)] $\Om$ is a compact metric space and ${\cal  F}$ is the completion of its Borel $\sigma$-algebra.
 \end{enumerate} 
 \begin{enumerate}
  \item[(H5'-b)] The mapping $\om \mapsto a(0,\om)$ is continuous on $\Om$.
  \end{enumerate} 
 \begin{enumerate}
\item[(H5'-c)] The group action of  $\R^d$ on $(\Om,{\cal F},\mathbb{P})$ defined by $a(\cdot,T_x \om)=a(x+\cdot,\om)$ defines a continuous action of $\R^d$ on $L^1(\Om)$. Namely, for every $\Phi\in L^1(\Om)$, $\Phi\circ T_x$ also belongs to $L^1(\Om)$ and moreover the mapping  $x\mapsto \Phi\circ T_x$ is continuous from $\R^d$ into $L^1(\Om)$.
\end{enumerate}
\begin{ex}
A class of examples satisfying assumptions (H1)--(H5), introduced in~\cite{armstrong2017quantitative}, 
can be constructed using a Poisson point process. Let us recall that a Poisson point process on a 
measurable space $(E,\mathcal{E})$ with intensity measure $\lambda$, is a random subset $\Pi$ 
of $E$ such that the following properties hold:
\begin{itemize}
      \item for every measurable set $A$ of $\mathcal{E}$, the number of points in $\Pi\cap A$, 
      denoted by $N(A)$, follows a Poisson law with mean $\lambda(A)$,
      \item for every pairwise disjoint measurable sets $A_1,\ldots,A_k$ of $\mathcal{E}$, the 
      random variables $N(A_1),\ldots,N(A_k)$ are independent.
\end{itemize}
  
We now consider the case where $\Pi$ is a Poisson point process on $\R^d$ with intensity 
measure $\lambda$ being the Lebesgue measure, defined on the Borel sets of $\R^d$, and  
$a_0$, $a_1$ are two matrices in the set 
\begin{equation*}
 \{ \widetilde{a}\in \R^{d\times d}_{\operatorname{sym}}:   \forall \xi \in \R^d, \; c_1|\xi|^2 
 \leq \xi\cdot \widetilde{a}\xi \leq c_2 |\xi|^2 \} \, ,
\end{equation*}
with $c_1,c_2>0$. We define a random field $a$ by setting, for every $x \in\R^d$, and every $\om \in \Om$,
\begin{equation*}
  a(x,\om) = \left\{
  \begin{aligned}
  a_0 & \text{ if } \operatorname{dist}(x,\Pi(\om)) \leq \frac{1}{2} \, ,\\
  a_1 & \text{ otherwise} \, .
  \end{aligned}\right. 
\end{equation*}
\begin{figure}
   \centering
   \includegraphics[scale=0.7]{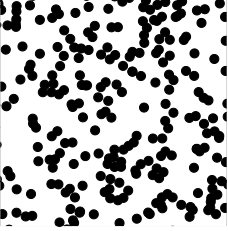}
\caption{A sample of the coefficient field defined by the homogeneous Poisson point cloud. 
The matrix $a$ is equal to $a_0$ in the black region and to $a_1$ in the white region~\cite{armstrong2017quantitative}.}
\end{figure}
Let us choose $0<\eps < \|a_0-a_1\|_{\R^{d\times d}}$. Then
\begin{equation*}
   \limsup_{y\rightarrow x}\mathbb{P}\left(\| a(y,\cdot)-a(x,\cdot)\|_{\R^{d\times d}} \geq\eps\right) = \mathbb{P}\left(\operatorname{dist}(x,\Pi)=\frac{1}{2}\right)= 0\, ,
\end{equation*}
and (H4) is verified. Hypotheses (H1), (H2) and (H3) are obviously satisfied.
\end{ex}
In the next three propositions, we prove that (H1)--(H5) imply (H5') in the case where $a$ 
is real valued for which the demonstration 
is simpler. The general case follows from direct modifications that we leave to the reader. 
(H5'-a) is established in Proposition~\ref{prop2}, (H5'b) in Proposition~\ref{continuous} and (H5'-c) in Proposition~\ref{prop5}.

Let us start with a technical lemma.
\begin{lem}\label{lem1}
Let $(\Om,\mathcal{A},\mathbb{P})$ be a standard probability space with $\mathcal{A}$ 
a complete $\sigma$-algebra and $b(x,\om)$, $x$ in $\R^d$, $\om$ in $\Om$, be 
a real bounded  random field. Let $\mathcal{N}$ be a dense countable subset of $\R^d$ and 
$\mathcal{F}\subset\mathcal{A}$ be the $\sigma$-algebra generated by $\{b(y,\cdot),\,y\in\mathcal{N}\}$.\\
Assume that for every $x$ in $\R^d$, and for every positive $\eps$,
\begin{equation*}
    \lim_{y\rightarrow x} \mathbb{P}(|b(y,\cdot)-b(x,\cdot)|\geq \eps) =0 \, .
\end{equation*}
Then $b$ is measurable with respect to $\mathcal{F}$.
\end{lem}
\begin{proof}
For simplicity, we only show that the set  $A=\{\om\in\Om:b(0,\om)>0\}$ is an element of $\mathcal{F}$. 
Let $(x_n)_{n\in\N}$ be a sequence of elements of $\mathcal{N}$ that converges to $0$. By assumption, for every $\eps>0$,
\begin{equation*}
\lim_{n\rightarrow\oo} \mathbb{P}\left(|b(x_n)-b(0)|>\eps\right)=0\, .
\end{equation*}
By a diagonal argument and up to extraction, we can assume that for every $n$ in $\N$,
\begin{equation*}
   \mathbb{P}\left(\left|b(x_n)-b(0)\right|>2^{-n}\right) \leq 2^{-n} \, .
\end{equation*}
Let us note $E=\limsup_n \left\{\om\in\Om : |b(x_n,\om) - b(0,\om)|>2^{-n}\right\}$. Borel-Cantelli lemma implies $\mathbb{P}(E)=0$. Now, for 
$\om\in \Om\setminus E$, we have $b(0,\om)= \liminf_{n} b(x_n,\om)$, so that 
\begin{equation*}
   \begin{aligned}
      A\setminus E & = \left\{ \om\in\Om\setminus E :  \liminf_{n\rightarrow\oo} b(x_n,\om) > 0 \right\}&\\
      & = \bigcap_{k>0} \left\{\om\in\Om:\liminf_{n\rightarrow\oo}b(x_n,\om)\geq \frac{1}{k}\right\}\setminus E \,.&   
      \end{aligned}
\end{equation*}
Hence, $A$ is an element of $\mathcal{F}$.
\end{proof}
We now build a compact metric space $K$ and a random process $b(\cdot,f)$ indexed by $f\in K$ with the same law as $a(\cdot,\om)$.

Let $\mathcal{N}$ be a dense countable subset of $\R^d$ containing $0$. 
Let $c_1,c_2$ in $\R$ be such that for every $x$ in $\R^d$ and $\om$ in $\Om$, $c_1 \leq a(x,\om) \leq c_2$ and let us set
  \begin{equation*}
  K = \left\{\{f(x)\}_{x\in\mathcal{N}}: \forall x \in \mathcal{N},\, f(x)\in[c_1,c_2]\right\} =  [c_1,c_2]^{\mathcal{N}}.
  \end{equation*}
  Let $\{\alpha_x\}_{x\in\mathcal{N}}$ be a summable family of positive numbers, we define a distance on $K$ by
  \begin{equation}
  d(f,g)=\sum_{x\in\mathcal{N}} \alpha_x \left|f(x)-g(x)\right|\,.\label{def_dist}
  \end{equation}
  According to Tikhonov theorem, $K$ equipped with the distance $d$ is a compact metric space.
  
  Now let us define a probability measure on the Borelians of $K$ by
  \begin{equation*}
  \mu(B)=\mathbb{P}\{\om\in\Om:\left(a(x,\om)\right)_{x\in\mathcal{N}}\in B\},
  \end{equation*}
  which is well defined according to Lemma~\ref{lem1}.
  Completing the Borelians with respect to $ \mu$, we obtain a $\sigma$-algebra $\mathcal{E}$. 
  The space $(K,\mathcal{E},\mu)$ is the canonical space for $\{a(x),\,x\in\mathcal{N}\}$.
  
  Finally, let us define $b$, for every $f$ in $K$, as follows: for every $x$ in $\mathcal{N}$,
  \begin{equation*}
  b(x,f) = f(x)\, .
  \end{equation*}
  Thanks to Lemma~\ref{lem1}, it follows that $b(x)$ is almost surely  uniquely defined for any $x$ in $\R^d$.
  We easily check that $b$ has the same law as $a$.
  We have established (H5'a), namely:
  \begin{prop}\label{prop2}\label{omegacompact} 
 The real-valued random field $b(x)$, $x\in\R^d$ over the probability space $(K, \mathcal{E},\mu)$ has the same law as $a$. Moreover $(K,d)$ is a compact metric space and ${\cal E}$ is the completion of its Borel sets with respect to $\mu$.
  \end{prop}

\noindent 
We check that (H5'-b) is satisfied.  
  \begin{prop} \label{continuous}
  Under the assumptions \emph{(H1)--(H5)} and with the notation introduced in the previous proposition, $b(0,\cdot)$ is continuous on $K$ for the distance $d$ defined by~\eqref{def_dist}.
  \end{prop}
  
\begin{proof}
  Let $f \in K$. Since $0\in \mathcal{N}$, we have for every $\widetilde{f}$ in $K$,
\[
  \alpha_0\left|b(0,f)-b(0,\widetilde{f})\right| 
  \leq 
   \sum_{x\in\mathcal{N}} \alpha_x \left|b(x,f)-b(x,\widetilde{f})\right| 
   =
   d(f,\widetilde{f})\, .
\]
  As $\alpha_0 >0$, it follows that $b(0,\cdot)$ is continuous in $f$.
\end{proof}

In order to prove (H5'-c) in Proposition~\ref{prop5} we first establish the following.
 \begin{lem} \label{densityofCf}
The space $C_f(K)$ of continuous functions on $K=[c_1,c_2]^{\cal N}$ that only depend on a finite number of variables is dense in $C(K)$. Similarly, $C_f(K)$ is dense in $L^p(K,\mu)$ for any $1<p<\oo$.
  \end{lem}
\begin{proof}
Let us enumerate ${\cal N}=\{x_1,x_2,\cdots\}$. Let $\Phi\in C(K)$ and $N>0$. For $f\in K$, we note $f_N=\Pi_N f$ the element of $K$ defined as $f_N(x_j)=f(x_j)$ for $j=1,\cdots, N$ and $f_N(x_j)=c_1$ for $j>N$.  The mapping $\Pi_N$ is continuous from $K$ into $K$ and only depends on the first $N$ variables. Next, we define $\Phi_N :=\Phi \circ \Pi_N$ and by composition $\Phi_N$ belongs to $C_f(K)$. Let $m$ be the modulus of continuity of $\Phi$. We have 
\begin{align*}
\|\Phi_N-\Phi\|_\oo &= \sup \left\{ |\Phi(f)-\Phi(\Pi_N f)| :f\in K\right\}\\
& \leq\,   \sup\left\{m\left(\sum_{x\in {\cal N}}  \alpha_x |f(x)-\Pi_Nf(x)|\right):\, f\in K\right\}\\
& \leq  m\left(|c_2-c_1| \sum_{x\in {\cal N}\setminus\{x_1,\cdots,x_N\}} \!\!\!\!\! \alpha_x\right)\ \quad \overset{N\uparrow\oo}\longrightarrow\ 0.
\end{align*}
We conclude that $C_f(K)$ is dense into $C(K)$. Eventually the density of $C_f(K)$ in $L^p(K)$ follows from that of $C(K)$ in $L^p(K)$ for $p>1$.
\end{proof}
Let us define the dynamical system $\{T_x: K\rightarrow K\}_{x\in\R^d}$ by
\begin{equation*}
   b(y,T_x f)=b(y+x,f),
\end{equation*}
for every $ f \in K$ and $x,y \in \R^d$. Let us notice that for every $x$ in $\R^d$, $T_x$ is uniquely defined thanks to Lemma~\ref{lem1}.\\
By definition, for every  $ x,y$ in $\R^d$, $ T_{x+y}=T_x\circ T_y$ and $T_0=\operatorname{Id}_K$.
Moreover, the stationarity assumption (H2) implies
  for every  event $A$ of $\mathcal{E}$, for every $ x$ in $\R^d$, 
  \begin{equation}
 \label{eq_stationarity}
   \mu(T_x^{-1}(A)) =  \mu(A) \, .
   \end{equation}
Eventually, (H3) gives that
 $T$ is ergodic:  if an event $A$ of $\mathcal{E}$ is  $T$-invariant, then $\mu (A) = 0$ or $ \mu(A) = 1$.  \medskip
 
\noindent  
Let us now check that $T$ complies to (H5'-c).
  \begin{prop}\label{prop5}
  Let $\Phi\in L^1(K)$, then $\Phi\circ T_x$ belongs to $L^1(K)$ for $x\in \R^d$ and 
  \begin{equation}\label{ContinuityTranslation}
\lim_{y \to x} \int_K \left|\Phi(T_{y}f)-\Phi(T_{x}f)\right|\, d\mu(f) = 0, \text{ for every $x\in\R^d$, and $\Phi\in L^1(K)$}.
  \end{equation}
  \end{prop}
    \begin{rem}\label{remContinuityTranslation}
With the same arguments (see the proof below), we may also prove that for any $\Phi\in L^p(K)$, the mapping $x\mapsto \Phi\circ T_x$ is continuous from $\R^d$ into $L^p(K)$.
  \end{rem}
  \begin{proof}
Let $\Phi\in L^1(K)$. By measurability of $T_x:K\to K$, we see that $\Phi\circ T_x$ is measurable, moreover by the stationarity property~\eqref{eq_stationarity}, we have \[\mu\{\Phi\circ T_x\in B\}=\mu\{\Phi\in B\}\] for every $x\in \R^d$ and every Borel set $B\subset \R$. In particular, $\Phi\circ T_x \in L^1(K)$ with $\|\Phi\circ T_x\|_{L^1(K)}=\|\Phi\|_{L^1(K)}$.\\
Let us now check the continuity of $x\in\R^d\mapsto \Phi\circ T_x$. By density of $C(K)$ in $L^1(K)$, we can assume that $\Phi$ is continuous (hence bounded) in $K$ and in fact by Lemma~\ref{densityofCf}, we can assume that $\Phi$ only depends on a finite number of variables. Eventually, by stationarity we only have to check the continuity at $x=0$. So, let us assume that for every $f$ in $K$, $\Phi(f)=\varphi(f(x_1),\cdots,f(x_N))$ with $\varphi\in C([c_1,c_2]^N)$ and $x_1,\cdots,x_N\in {\cal N}$. Let  $\eta>0$, for $y\in \R^d$ we denote 
\[
K_1(y)=\{f\in K: \left| (f(x_j+y)-f(x_j))_{j=1,\cdots,N}\right|<\eta\},\qquad K_2(y)=K\setminus K_1(y).
\]
Decomposing the integration on $K$ over $K_1$ and $K_2$, we have
\[
\int_K \left|\Phi(T_{y}f)-\Phi(f)\right|\, d\mu(f)  \leq m(\eta) +2 \|\varphi\|_\oo\mu(K_2(y))\,,
\]
where $m$ is the modulus of continuity of $\varphi$. Now, by assumption $(H5)$,  $\mu(K_2(y))$ tends to $0$ as $y$ tends to 0 and since $\eta>0$ is arbitray, we see that the integral in the left hand side also goes to 0.
  \end{proof}

From now on, we assume that  $(\Om,\mathcal{F},\mu)$  is the canonical space associated with $a$, and that it verifies (H1)-(H5'). As for every $x$ in $\R^d$ and $\omega$ in $\Omega$, $a(x,\omega)=a(0,T_x\omega)$, for simplicity we will denote $a(x,\omega)=a(T_x\omega)$. By (H5'), it holds that   $a$ is in the space of continuous functions defined on $\Om$, which will be denoted $C(\Om)$.

\begin{rem}\label{compact-separ}
In Section~\ref{ss2scaleconv} below, we introduce the notion of two-scale convergence and use test functions in $L^2(\mathcal{D}\times\Om)$ where $\mathcal{D}$ is a bounded domain of $\R^d$. In the theory developed thereafter, we use the continuous embeding of $C_c(\mathcal{D}\times \Om)$ into $L^2(\mathcal{D}\times\Om)$, which is a consequence of the finiteness of the measure $\lambda\otimes\mu$ in $\mathcal{D}\times\Om$.  The following fact is more crucial: since $\Om$ is compact, the Banach space $(C_c(\mathcal{D}\times \Om),\|\cdot\|_\oo)$ is separable, so there exists a countable subset $\Gamma\subset C_c(\mathcal{D}\times \Om)$ which is dense in $(C_c(\mathcal{D}\times \Om),\|\cdot\|_\oo)$ and in $L^2(\mathcal{D}\times\Om)$.

\end{rem}

  We end this section by recalling the Birkhov ergodic theorem, introduced in~\cite{birkhoff1931proof}, that plays a prominent role in the analysis. We first check that the quantities involved are well defined.
\begin{lem}
Let $p\geq 1$ and $u$ be a function of $L^p(\Om)$. Then, $\mu$-almost surely, $x\mapsto u(T_x\om)$ is in $L^p_{\text{loc}}(\R^d)$.
\end{lem}
\begin{proof}
 For every bounded Borel set $A$ of $\R^d$, 
\begin{eqnarray*}
 \int_{\Om}\int_A |u(T_x\om)|^pdxd\mu(\om)& =& \int_A\int_{\Om}|u(T_x\om)|^pd\mu(\om)dx\\
 &=& |A|\int_{\Om}|u(\om)|^pd\mu(\om)\   < \ \oo\, ,
\end{eqnarray*}
according to Fubini's theorem and thanks to the stationarity of $a$. Therefore, $\int_A |u(T_x\om)|^pdx$ is bounded $\mu$-almost surely.
\end{proof}

\begin{thm} \emph{[Birkhov ergodic theorem]}. 
Let $f$ be a function of $  L^1(\Om)$. Then, for $\mu$-almost every $\widetilde{\om}$ in $\Om$ and for every bounded Borel set $A$ of $\R^d$,
\begin{equation}
  \frac{1}{t^d|A|}\int_{tA} f(T_x\widetilde{\om})\,dx \xrightarrow[t\rightarrow\oo]{} \int_{\Om} f(\om)\,d\mu(\om) =\E(f)\, . \label{birkhoveq}
\end{equation}
\label{birkhov}
\end{thm}
 
 \subsection{$L^2$ Two-scale convergence}\label{ss2scaleconv}
Let us start by noticing that the Birkhov ergodic theorem presented above is not 
sufficient to obtain results valid almost surely for all functions $f$, and thus, cannot be sufficient to obtain an homogenization theorem with almost sure convergence of the solution.  
One of the reasons is the fact that the set of 
$\widetilde{\omega}$ for which the convergence hold depends on $f$. Therefore, 
we introduce the following definition.
     
\begin{definition}  
Let  $\widetilde{\om}\in\Om$. We say that $\widetilde{\om}$ is a  \emph{typical trajectory}, if, 
\begin{equation*}
  \lim_{t\rightarrow\oo} \frac{1}{t^d|A|} \int_{tA} g(T_x\widetilde{\om}) \,dx = 
  \int_{\Om}g(\om)\,d\mu (\om) = \E(g),
\end{equation*}
for every bounded Borelian $A\subset \R^d$ with $|A|>0$ and every $g$ in $C(\Om)$. 
\end{definition}
  
\begin{prop}
  Let $\widetilde{\Om}$ be the set of typical trajectories. Then $\mu(\widetilde{\Om}) = 1$.
\label{GeneralizedBirkhov}  
\end{prop}

\begin{proof} We first notice that the compactness of $\Om$ entails that 
$C(\Om)$ (endowed with the norm $\|g\|_{\oo}=\sup_{\Om}|g|$) is separable. 
We thus consider 
\[
\Gamma=\{g_k,\,k\in\N\}
\]  
a dense countable subset of  
$C(\Om)$. According to Birkhov ergodic theorem, for every $k$ in $\N$, 
there exists $\Om_k$ in $\mathcal{F}$ such that $ \mu(\Om_k)=1$,  
and $\Om_k \subset\Om_{k-1}\subset \cdots\subset\Om_0$ and,  
for every $\om'\in\Om_k$, 
\begin{equation*}
  \lim_{t\rightarrow \oo} \frac{1}{t^d|A|}\int_{tA} g_k(T_x\om') \,dx = 
  \int_{\Om}g_k(\om)\,d\mu (\om) = \E( g_k),
\end{equation*}
for every bounded Borel set $A$ with $|A|>0$. Considering $\Omega'=\cap_{k\in\N}\Omega_k$, 
we have $\mu(\Omega')=1$, and it only remains to show that $\Omega' \subset \widetilde{\Omega}$.\\
Let $\omega'\in\Omega'$, $g\in C(\Omega)$, $\epsilon>0$ and $g'\in\Gamma$ 
such that $\|g'-g\|_{\infty}<\epsilon$. For every bounded Borel set  
$A\subset \R^d$ and $t>0$,
\begin{align*}
  \MoveEqLeft[15] \left| \frac{1}{t^d|A|}\int_{tA} g(T_x\omega') \,dx -\E(g)\right| \\
\MoveEqLeft[14]\leq\left| \frac{1}{t^d|A|}\int_{tA} [g-g'](T_x\omega') \,dx \right|
+\left| \frac{1}{t^d|A|}\int_{tA} g'(T_x\omega') \,dx - \E(g')\right| 
   + \left|\E(g'-g)\right|\\
   \MoveEqLeft[14]\leq 2\|g'-g\|_\infty +\left| \frac{1}{t^d|A|}\int_{tA} g'(T_x\omega') \,dx - \E(g')\right|\\
\MoveEqLeft[14]\leq  3 \epsilon, \, 
\end{align*}
for $t$ large enough since $g'\in \Gamma$ and  $\omega'\in \Omega'$. Therefore, 
\begin{equation*}
  \lim_{t\rightarrow \oo} \frac{1}{t^d|A|}\int_{tA} g(T_x\om') dx = 
  \int_{\Om}g(\om)d\mu (\om) = \E(g) \, ,
\end{equation*}
which gives $\Om'\subset\widetilde{\Om}$. We deduce 
that $\mu(\widetilde{\Om})\geq \mu(\Om')=1$.
\end{proof}

\begin{prop}
\label{meanvalue}
\emph{[Mean-value property.]} Let $g$ be a function in  $C(\Om)$. Then, 
for every $\varphi$ in $C_c(\R^d)$, for every $\widetilde{\om}$ in $\widetilde{\Om}$,
\begin{equation*}
  \lim_{\eps\downarrow0} \int_{\R^d} \varphi(x)g(T_{x/\eps}\widetilde{\om})dx = 
  \int_{\R^d}\varphi(x)dx\int_{\Om} g(\om)d\mu(\om) = \E(g)\int_{\R^d}\varphi(x)dx\, .
\end{equation*}
\end{prop}

\begin{proof}
We use Proposition~\ref{GeneralizedBirkhov} to get the result for a simple function 
$\varphi =\sum_{i=1}^N a_i\mathbf{1}_{A_i}$ where $(A_i)_{1\leq i \leq N}$ are bounded Borel sets of $\R^d$,
and $\mathbf{1}_{A_i}$ denotes the characteristic function of $A_i$. 
The conclusion comes from the approximation of any function 
$\varphi\in C_c(\R^d)$ by simple functions in the 
$L^1(\R^d)$ norm.
%
 \end{proof}
 
We now state the two-scale convergence definition and prove the main compactness theorem.
  
\begin{definition}
\label{def}
Let $\widetilde{\om}\in\widetilde{\Om}$ be fixed, let $\{v^\eps\}_{\eps\in I}$ be a family of elements of $L^2(\mathcal{D})$ indexed by $\eps$ in a set $I\subset (0,+\oo)$ with $0\in \overline I$ and let $v\in L^2(\mathcal D \times \Om)$. Let $(\eps_k)_{k\geq0}\subset I$ be a decreasing sequence converging to $0$, we say that the subsequence $(v^{\eps_k})$ 
\emph{weakly two-scale converges} to $v$ 
if, for every $ \varphi$ in $C_c^{\oo}(\mathcal{D})$ and every $b$ in $C(\Om)$,
\begin{equation*}
   \lim_{k\to \infty} \int_{\mathcal{D}} v^{\eps_k}(x)\varphi (x) b (T_{x\slash\eps_k}\widetilde{\om})\,dx 
    = \int_{\mathcal{D}} \int_{\Om} v(x,\om)\varphi (x) 
    b(\om) d \mu (\om) \,dx \, .
\end{equation*}
We write
\begin{equation*}
  v^{\eps_k}\in L^2(\mathcal D) \overset{2}{\rightharpoonup} v\in L^2(\mathcal D\times \Om)  . 
\end{equation*}
\end{definition}
It is worth noticing that this definition of two-scale convergence, and thus the limit, 
depends on the choice of $\widetilde{\om}$ in $\widetilde{\Om}$. From now on, we assume that $\widetilde{\om} \in \widetilde{\Om}$ is fixed.
 
The main result of this subsection is the following theorem. It is the stochastic 
equivalent of the two-scale compactness theorem provided in~\cite{allaire1992homogenization} for periodic homogenization.

\begin{thm}\label{compacite}
Let $\{v^{\eps}\}_{\eps\in I}$ be a bounded family in $L^2(\mathcal{D})$, with $I$ as in the above definition. 
Then, there exist a sequence  $(\eps_k)_{k\geq0}$ in $I^\N$ that tends to zero, and 
$v^0$ in $L^2(\mathcal{D}\times \Om)$ such that $(v^{\eps_k})_k$ weakly 
two-scale converges to $v^0$.
\end{thm}
  
\begin{proof}
Let  $K =  \{\varphi\, b, \, \varphi\in C_c^{\oo}({\cal D}),b\in C(\Om)\} $ be the set 
of test functions and $\left<K\right>$  be its linear span, {\it i.e.} $\left<K\right>$ is the tensor product $C_c^{\oo}({\cal D})\otimes C(\Om)$. Using the Cauchy-Schwarz inequality and the boundedness of $\{v^{\eps}\}$ in 
$L^2(\mathcal{D})$, it holds that, for every $\eps>0$ and every  $\Phi$ in $\left<K\right>$,
\begin{equation*}
  \left|\int_{\mathcal{D}} v^{\eps}(x)\Phi(x,T_{x/\eps}\widetilde{\om})dx\right| 
  \leq C \left( \int_{\mathcal{D}} \Phi^2(x,T_{x/\eps}\widetilde{\om})dx\right)^{1/2}.
\end{equation*}
Decomposing $\Phi$ as $\Phi=\sum_{p=1}^N \Phi_p$ with $\Phi_1,\cdots,\Phi_N\in K$, and using Proposition~\ref{meanvalue}, it is easily checked that
\begin{equation*}
  \lim_{\eps\to 0} \int_{\mathcal{D}} \Phi^2(x,T_{x/\eps}\widetilde{\om})dx 
  =  \int_{\mathcal{D}\times\Om} \Phi^2(x,\om)dxd\mu(\om)\, .
\end{equation*}
Therefore,
\begin{equation}
  \limsup_{\eps\to 0} \left|\int_{\mathcal{D}} v^{\eps}(x)\Phi(x,T_{x/\eps}
  \widetilde{\om})dx\right| \leq C \|\Phi\|_{L^2(\mathcal{D}\times\Om)} \, . 
\label{limsup}
\end{equation}
In particular, for every $\Phi\in\left<K\right>$, the family $\{\int_{\mathcal{D}} v^{\eps}(x)\Phi(x,T_{x/\eps}\widetilde{\om})dx\}_{\eps>0}$ is bounded in $\R$.
Recalling Remark~\ref{compact-separ}, we pick a countable subset $\Gamma\subset \left<K\right>$ which is both dense in $C(\mathcal{D}\times \Om)$ and in $L^2(\mathcal{D}\times \Om)$. Using a diagonal process, there exists a sequence $(\eps_k)_k$ that tends to zero, such that for every $\Psi$ in $\Gamma$,
 \begin{equation}\label{l(Psi)}
 \lim_{k\rightarrow\oo} \int_{\mathcal{D}} v^{\eps_k}(x)\Psi\left(x,T_{x/\eps_k}\widetilde{\om}\right)dx = \ell(\Psi) \, ,
 \end{equation}
for some real number $\ell(\Psi)$.  By linearity, this relation extends to the linear span $\left<\Gamma\right>$ of $\Gamma$, the function $\ell$ is a linear form on $\left<\Gamma\right>$ and by~\eqref{limsup}, for every $\Psi\in\left<\Gamma\right>$, there holds 
\begin{equation}\label{estiml(Psi)}
\ell(\Psi)\leq C \|\Psi\|_{L^2(\mathcal{D}\times\Om)}.
\end{equation}
By density of $\Gamma$ in $L^2(\mathcal{D}\times \Om)$, we see that $\ell$ uniquely extends as a continuous linear map $\ell:L^2(\mathcal{D}\times\Om)\to\R$ satisfying~\eqref{estiml(Psi)}. Now, let $\Phi\in K$ and let $\eta>0$, there exists $\Psi\in \Gamma$ such that $\|\Psi-\Phi\|_\oo<\eta$. Using~\eqref{limsup}, we compute
\begin{multline*}
\left|\int_{\mathcal{D}} v^{\eps_k}(x)\Phi\left(x,T_{x/\eps_k}\widetilde{\om}\right)dx - \ell(\Phi)\right|
\leq \left|\int_{\mathcal{D}} v^{\eps_k}(x)(\Phi-\Psi)\left(x,T_{x/\eps_k}\widetilde{\om}\right)dx \right| \\+\left|\int_{\mathcal{D}} v^{\eps_k}(x)\Psi\left(x,T_{x/\eps_k}\widetilde{\om}\right)dx - \ell(\Psi)\right| + |\ell(\Psi-\Phi)|\\
\leq C\sqrt{|\mathcal D|} \; \|\Phi-\Psi\|_\oo +\left|\int_{\mathcal{D}} v^{\eps_k}(x)\Psi\left(x,T_{x/\eps_k}\widetilde{\om}\right)dx - \ell(\Psi)\right| + C\|\Psi-\Phi\|_{L^2}\\
\leq  2C\sqrt{|\mathcal D|}\, \eta+\left|\int_{\mathcal{D}} v^{\eps_k}(x)\Psi\left(x,T_{x/\eps_k}\widetilde{\om}\right)dx - \ell(\Psi)\right|\,.
\end{multline*}
Since $\Psi\in \Gamma$, the last term tends to $0$ as $k$ tends to infinity and since $\eta>0$ is arbitrary, we obtain that~\eqref{l(Psi)} holds for every $\Psi\in K$.\\
Eventually, since $\ell$ is a continuous linear form on the Hilbert space $L^2(\mathcal{D}\times\Om)$, by Riesz representation theorem, there exists $v^0$ in $L^2(\mathcal{D}\times\Om)$ such that for every $\Phi$ in $L^2(\mathcal{D}\times\Om)$, 
 \begin{equation*}
 \ell(\Phi) = \int_{\mathcal{D}\times\Om} v^0(x,\om)\Phi(x,\om)dxd\mu(\om)\,.
 \end{equation*}
 This entails, in particular that for every $\varphi$ in $C_c^{\oo}(\mathcal{D})$ and $b$ in $C(\Om)$,
 \begin{equation*}
 \lim_{k\rightarrow\oo} \int_{\mathcal{D}}v^{\eps_k}(x)\varphi(x)b(T_{x/\eps_k}\widetilde{\om})dx = \ell(\varphi\, b) = \int_{\mathcal{D}\times \Om} v^0(x,\om)\varphi(x)b(\om)dxd\mu(\om) \, .
 \end{equation*}
  \end{proof}
  
Two-scale convergence is mostly a weak notion. A corresponding strong two-scale 
convergence exists that allows us to use weak-strong convergence properties as stated in the following. 
  
\begin{definition}
The sequence $(v^{\eps_k})_{k\geq0}$ of $L^2(\mathcal{D})$ is said to  
\emph{strongly two-scale converge} to a function $v^0$ in $L^2(\mathcal{D}\times\Om)$ 
as $\eps$ goes to $0$ if $(v^{\eps_k})$ weakly two-scale converges to  $v^0$ and if
\begin{equation*}
    \lim_{\eps_k\to 0} \|v^{\eps_k}\|_{L^2(\mathcal{D})} = \|v^0\|_{L^2(\mathcal{D}\times \Om)} \, .
\end{equation*}
\end{definition}
\begin{rem}
\label{remstrongconv}
An important example of strong convergence is the following. If $(v^{\eps_k})$ converges towards some function $\bar{v}$ strongly in $L^2(\mathcal{D})$, then, by definition,  $(v^{\eps_k})$ weakly two-scale converges towards $v^0$ defined as $v^0(x,\om):=\bar{v}(x)$.
Moreover,  
\[
\|v^0\|_{L^2(\mathcal{D}\times \Om)}= \|\bar{v}\|_{L^2(\mathcal{D})} = \lim_{\eps_k\downarrow 0} \|v^{\eps_k}\|_{L^2(\mathcal{D})}\,,
\]
and $(v^{\eps_k})$  strongly two-scale converges towards $v^0$.
\end{rem}

Strong two-scale convergence allows us to have a two-scale version of the weak-strong convergence principle.

\begin{prop}\label{cv2echforte}
Let $\{v^{\epsilon}\}_{\epsilon>0}$ and $\{ u^{\epsilon} \}_{\epsilon>0}$ be two families of 
functions of $L^2(\mathcal{D})$. If  $\{v^{\epsilon}\}_{\epsilon>0}$ strongly two-scale 
converges to $v^0$ and $\{ u^{\epsilon} \}_{\epsilon>0}$ weakly two-scale converges to $u^0$, 
for every $ \varphi$ in $C_c^{\infty}(\mathcal{D})$ and $b$ in $C(\Omega)$,
\begin{equation*}
   \lim_{\eps_k \to 0}  \int_{\mathcal{D}} u^{\eps_k}(x) v^{\eps_k}(x)\varphi (x) b (T_{x/\eps_k}\widetilde{\om})\,dx 
   = \int_{\mathcal{D}} \int_{\Om} u^0(x,\om) v^0(x,\om)\varphi (x) 
    b(\om) \,d \mu (\om) \,dx \, .
\end{equation*}
\end{prop}
  
\begin{proof}
Let $\{u^{\eps}\}_{\eps>0}$, $\{v^{\eps}\}_{\eps>0}$, $u_0$, $v_0$ and $(\eps_k)$ satisfying the assumptions of the proposition. For simplicity, we drop the subscript $k$ in the proof below. 
Let $\delta>0$, there exists $\Phi\in  C_c^\oo(\mathcal{D}) \otimes C(\Om)$ such that 
\[
 \lt\|\Phi-v^0\rt\|_{L^2(\mathcal{D}\times\Omega)}^2 \leq \delta^2
\,.
\]
Since $v^\eps\overset{2}{\rightharpoonup} v^0$ with strong convergence, and using Proposition \ref{meanvalue}, there exists $\eps_\delta$ such that for $0<\eps<\eps_\delta$,
\begin{align*}
 \left|\int_{\mathcal{D}} \Phi(x,T_{x/\eps}\widetilde{\om})^2 dx-\|\Phi\|_{L^2(\mathcal{D}\times\Om)}^2\right| &\leq \delta^2 \,,\\
 \left|\int_{\mathcal{D}}v^{\eps}(x)\Phi(x,T_{x/\eps}\widetilde{\om})dx-\int_{\mathcal{D}}\int_{\Om}v^0(x,\om)\Phi(x,\om)\,dx\,d\mu(\om)\right|& \leq \delta^2\,,\\
 \left| \int_{\mathcal{D}} \left(v^\eps (x)\right)^2\,dx - \int_{\mathcal{D}\times \Om} \left(v^0(x,\om) \right)^2\,dx\,d\mu(\om)\right| &\leq \delta^2\,.
\end{align*}
Combining the preceding estimates leads to
\begin{equation}\label{strconv}
  \int_{\mathcal{D}} \left(v^{\eps}(x)-\Phi(x,T_{x/\eps}\widetilde{\om})\right)^2dx \leq 5\delta^2\,.
\end{equation}
Now, let $\varphi \in C_c^{\oo}(\mathcal{D})$ and $b \in C(\Om)$, we write for $\eps>0$, 
\begin{multline*}
  \int_{\mathcal{D}} u^{\eps}(x)v^{\eps}(x)\varphi(x)b(T_{x/\eps}\widetilde{\om})\,dx  = \int_{\mathcal{D}} u^{\eps}(x)\Phi(x,T_{x/\eps}\widetilde{\om})\varphi(x)b(T_{x/\eps}\widetilde{\om})\,dx \\
+ \int_{\mathcal{D}}u^{\eps}(x)\left(v^{ \eps}(x)-\Phi(x,T_{x/\eps}\widetilde{\om})\right)\varphi(x)b(T_{x/\eps}\widetilde{\om})\,dx \, .
\end{multline*}
Since $u^\eps\overset{2}{\rightharpoonup} u^0$, there exists $\eps_\delta'\in(0,\eps_\delta]$ such that for every $0<\eps<\eps'_\delta$,
\[
  \left|\int_{\mathcal{D}} u^{\eps}(x)\Phi(x,T_{x/\eps}\widetilde{\om})\varphi(x)b(T_{x/\eps}\widetilde{\om})\,dx -\int_{\mathcal{D}\times\Om}u^0\Phi\,\varphi\, b\,dx\,d\mu\right| \leq \delta \, .
\]
We obtain, for $0<\eps<\eps'_\delta$\,,
\begin{align*}
  \MoveEqLeft[7]\left|\int_{\mathcal{D}} u^{\eps}(x)v^{\eps}(x)\varphi(x)b(T_{x/\eps}\widetilde{\om})\,dx -\int_{\mathcal{D}\times\Om}u^0v^0\varphi\, b\,dx\,d\mu\right| \\
& \leq \delta + \left|\int_{\mathcal{D}\times\Omega}u^0\varphi b(\Phi-v^0)dxd\mu\right| 
\\&+ \left|\int_{\mathcal{D}}u^{\eps}(x)\left(v^{ \eps}(x)-\Phi(x,T_{x/\eps}\widetilde{\om})\right)\varphi(x)\,b(T_{x/\eps}\widetilde{\om})\,dx\right| \\
&\leq \delta\left(1+\|u^0\varphi b\|_{L^2(\mathcal{D}\times\Omega)}\right)\\&+\left|\int_{\mathcal{D}}u^{\eps}(x)\left(v^{ \eps}(x)-\Phi(x,T_{x/\eps}\widetilde{\om})\right)\varphi(x)\,b(T_{x/\eps}\widetilde{\om})\,dx\right|\, .
\end{align*}
To estimate the last term, we use the Cauchy-Schwarz inequality to get, 
\begin{multline*}
  \left|\int_{\mathcal{D}}u^{\eps}(x)\left(v^{ \eps}(x)-\Phi(x,T_{x/\eps}\widetilde{\om})\right)\varphi(x)b(T_{x/\eps}\widetilde{\om})dx  \right| \\
  \leq C \|u^{\eps}\|_{L^2(\mathcal{D})}\left(\int_{\mathcal{D}} \left(v^{\eps}(x)-\Phi(x,T_{x/\eps}\widetilde{\om})\right)^2dx \right)^{1/2}\leq  C'\delta,\end{multline*}
where we have used  the boundedness of $(u_{\eps})_{\eps}$ in $L^2(\mathcal{D})$ and~\eqref{strconv} to get the last inequality.
Consequently, for every $0<\eps<\eps_\delta'$,
\begin{equation*}
  \left|\int_{\mathcal{D}} u^{\eps}(x)v^{\eps}(x)\varphi(x)b(T_{x/\eps}\widetilde{\om})\,dx -\int_{\mathcal{D}\times\Om}u^0v^0\varphi\, b\,dx\, d\mu\right| \leq C''\delta\, .
\end{equation*}
This proves the proposition.
\end{proof}

 The key ingredient in the applications of two-scale convergence to homogenization problems, as introduced for the first time in~\cite{allaire1992homogenization}, is the use of a two-scale compactness theorem on $H^1(\mathcal{D})$.  This compactness is used to pass to the two-scale limit in the integral formulation of the equations for both the solution and its gradient. In order to extend such property to our setting, we first  need to define a space $H^1(\Om)$.
 
  \subsection{Construction of $H^1(\Om)$}
 
\begin{definition}
Let $u$ be a function of $L^2(\Om)$. We say that $u$ is differentiable at $\om$ in $\Om$ if, for every $i$ in $\{1,\ldots,d\}$, the limit
\begin{equation}\label{defderivative}
    \lim_{\delta \rightarrow 0} \frac{u(T_{\delta e_i}\om) -u(\om)}{\delta}=: (D_i u)(\om)
  \end{equation}
  exists, where $(e_1, ..., e_d)$ denotes the canonical basis of $\R^d$. In this case, we note $D_\om u=(D_1u,\cdots,D_du)$.
\end{definition}

\begin{lem}\label{lien_derivees}
Let $u$ be a function of $L^2(\Om)$, differentiable at every point of $\Om$. Then, for every $x$ in $\R^d$ and $i$ in $\{1,\ldots,d \}$,
\begin{equation*}
\frac{\partial}{\partial x_i} [u (T_x\om)] = (D_i u)(T_x\om) \, .
\end{equation*}
\end{lem}

\begin{proof}
For every $x$ in $\R^d$ and $i$ in $\{1,\ldots,d \}$,
\begin{align*}
  \frac{\partial}{\partial x_i} [u(T_x\om)] &= \lim_{\delta\to0}\frac{u(T_{x + \delta e_i}\om) - u(T_x \om)}{\delta} \\
  & =  \lim_{\delta\to0}\frac{u(T_{\delta e_i}(T_x\om)) - u(T_x \om)}{\delta} 
    = (D_i u)(T_x\om)\, .
\end{align*}
\end{proof}

  \begin{definition}
   Let $C^1(\Om)$ be the set of functions $u$ of $\Om$ that are continuous and differentiable at every $\om$ in $\Om$ and such that for every $i$ in $\{ 1, \ldots,d\}$, the function $D_i u$ is continuous on $\Om$.
  \end{definition}

  \begin{lem}\label{lemC1denseL2}
    $C^1(\Om)$ is dense in $L^2(\Om)$.\label{densite}
  \end{lem}

  \begin{proof}
  Since $C(\Om)$ is dense in $L^2(\Om)$ (see~\cite[Theorem 3.14]{rudin1964principles}), it suffices to show that $C^1(\Om)$ is dense in $C(\Om)$ for the $L^2$-norm.\\
    Let $\rho\in  C^{\oo}_c(\R^d,\R_+)$ such that $\int_{\R^d} \rho=1$. We define a standard family of approximations of  unity $\{\rho^\delta\}_{\delta>0}$ as $\rho^\delta (x)=(1/\delta)^d\rho(x/\delta)$ for $x\in\R^d$,\ $\delta>0$. Now, let $\varphi$ be a function in $ C(\Om)$, we consider its mollifications $\{\varphi^\delta\}$ defined for $\delta>0$ and $\om\in \Om$ as,
  \[
 \varphi^{\delta}(\om):= \int_{\R^d}\rho^\delta\left(y\right)\varphi(T_{(-y)}\om)dy \,.
  \]
 We easily check that $\varphi^\delta\in C^1(\Om)$ with 
 \[
 D_i\varphi^\delta(\om)=\int_{\R^d}\partial_{x_i}\rho^\delta\left(y\right)\varphi(T_{(-y)}\om)dy \, .
 \]
Using assumption (H5'-c) (see also Remark~\ref{remContinuityTranslation}) it is also standard to check that we have 
  \[
  \|\varphi^\delta -\varphi\|_{L^2(\Om)}\ \overset{\delta \downarrow 0}\longrightarrow\ 0.
  \]
This proves the lemma.
  \end{proof}\

 \begin{lem}
 \label{lemipp}
 Let $w$ in $C^1(\Om)$, then $\E({D_i w})=0$ for $i$ in $\{1,\cdots,d\}$. In particular, for $u$,$v$ in $C^1(\Om)$, there holds
 \begin{equation}\label{ipp}
 \int_{\Om} D_i u\, v \, d\mu +\int_\Om u\, D_i v\, d\mu\, =\, 0.
 \end{equation}
 \end{lem}
 \begin{proof}
 Let $w$ in $C^1(\Om)$ and let $\chi$ in $C^\oo_c(\R^d)$ such that $\int\chi \, dx=1$, we have by Proposition~\ref{meanvalue},
 \[
 \E({ D_iw }) = \lim_{\eps\downarrow0} \int_{\R^d} \chi(x)D_i w(T_{x/\eps}\widetilde{\om})dx.
 \]
 On the other hand, for $\eps>0$, we have, using Lemma~\ref{lien_derivees} and integrating by parts, 
\begin{align*}
  \int_{\R^d} \chi(x)D_i w(T_{x/\eps}\widetilde{\om})\, dx& = \eps \int_{\R^d} \chi(x)\dfrac \partial{\partial x_i} \lt[w(T_{x/\eps}\widetilde{\om})\rt]\,dx\\
  &=-\eps\int_{\R^d} \dfrac{\partial\chi}{\partial x_i}(x)\, w(T_{x/\eps}\widetilde{\om})\,dx.
 \end{align*}
 By Proposition~\ref{meanvalue} again, the last expression tends to 0 as $\eps$ goes to $0$ which proves the first part of the lemma.\\
Now if $u$, $v$ are two functions of $C^1(\Om)$, we easily see that $w:=uv$ is also in $C^1(\Om)$ with the product rule $D_i w=D_iu\, v+u\,D_iv$. The identity~\eqref{ipp} then follows from $\E({D_iw})=0$. 
\end{proof}

  We now define the weak derivatives of a function $u\in L^2(\Om)$ by duality.
  
 \begin{definition}
Let $u$ in $L^2(\Om)$, $i$ in $\{1,\cdots,d\}$ and $w_i$ in $L^2(\Om)$; we say that $u$ admits $w_i$ as weak derivative in the $i^{\text{th}}$ direction if
\[
-\int_\Om uD_iv\, d\mu =\int_\Om w_i v\,d\mu\quad\text{for every }v\in C^1(\Om).
\]
We note $D_iu=w_i$ the weak derivative (which is uniquely defined) in $L^2(\Om)$, noticing that both definitions of $D_i$ coincide for $C^1$ random variable.
If $u$ admits weak derivatives in $L^2(\Om)$ in all directions, we say that $u$ belongs to $H^1(\Om)$ and we note $D_\om u=(D_1u,\cdots,D_du)$.
Obviously, $H^1(\Om)$ is a vector space. We define an inner product on $H^1(\Om)$ as 
\[
\lt(u,v\rt)_{H^1(\Om)}:=(u,v)_{L^2(\Om)} +(D_\om u,D_\om v)_{L^2(\Om)^d}.
\]
  \end{definition}
%
%
%
\vskip 0.1 in

We easily see that $H^1(\Om)$ shares several properties with the usual Sobolev space $H^1(U)$ when $U$ is a bounded open set of $\R^N$. First, using the isometric embedding $j: u\in H^1(\Om)\mapsto (u,D_\om u)\in L^2(\Om)\times L^2(\Om,\R^d)$, the closed graph theorem implies that $j(H^1(\Om))$ is a closed subspace of $L^2(\Om)\times L^2(\Om,\R^d)$, hence $(H^1(\Om),(\cdot,\cdot)_{H^1(\Om)})$ is a Hilbert space. Moreover, using the mollifying procedure of Lemma~\ref{lemC1denseL2}, we may prove that $C^1(\Om)$ is dense in $H^1(\Om)$ and the imbedding 
\[
C^1(\Om)\hookrightarrow H^1(\Om)
\]
is continuous.
Eventually, notice that for $u$ in $L^2(\Om)$, if there exists $C\geq 0$ such that 
\begin{equation}
\label{caractH1}
\lt\|u(T_{\delta e_i}\cdot) -u\rt\|_{L^2(\Om)}\leq C\delta \quad \text{ for }\delta>0\text{ and }i\text{ in }\{1,\cdots,d\},
\end{equation}
then $u$ is in $H^1(\Om)$. This is a classical consequence of~\eqref{defderivative}, of the Lebesgue dominated convergence theorem and of the Riesz representation theorem. Gathering theses facts, we have:
\begin{prop}
$H^1(\Om)$ is a separable Hilbert space; its subspace $C^1(\Om)$ is dense. Moreover, a function $u$ in $L^2(\Om)$ belongs to $H^1(\Om)$ if and only if~\eqref{caractH1} holds true.
\end{prop}  
  \begin{lem}  \label{grad}
   Let $u$ be a function in $H^1(\Om)$ and let us denote by $v_{\om}: x \mapsto u(T_x \om)$ and $z_{\om}: x \mapsto (D_{\om} u )(T_x\om)$.  Then, $\mu$-almost surely, $v_{\om}$ belongs to $H^1_{\text{loc}}(\R^d)$ and $z_{\om} = \nabla_x v_{\om}$ almost everywhere on $\R^d$.
  \end{lem}
  
  \begin{proof}
  Let us proceed by density. Let $(u_k)_k$ be a sequence in  $(C^1(\Om))^{\N}$ such that $u_k$ converges to $u$, and $D_{\om} u_k$ converges to $D_{\om}u$ in $L^2(\Om)$.

  For every  bounded Borel set $A$ of $\R^d$,
  \begin{equation*}
    \int_{\Om} \int_A \left(u_k(T_x\om)-v_{\om}(x)\right)^2\,dx\,d\mu(\om) = |A|\int_{\Om} \left(u_k(\om)-u(\om)\right)^2d\mu(\om) \, ,
    \end{equation*}
    because of stationarity. The right-hand side term converges to $0$ as $k$ goes to infinity by definition of $(u_k)_k$. As a consequence, thanks to a diagonal argument, 
    and up to extraction, $u_k(T_x\om)$ converges to $v_{\om}$ in $L^2_{\text{loc}}(\R^d)$, almost surely on $\Om$. \\
   We have similarly for $i$ in $\{1,\ldots,d \}$,
\[
   \int_\Om \int_A \left(\frac{\partial}{\partial x_i}u_k(T_x\om)-z_{\om,i}(x)\right)^2\,dx\,d\mu(\om) =|A|\int_{\Om} \left(D_iu_k(\om)-D_iu(\om)\right)^2d\mu(\om) \, ,
\]
   according to Lemma~\ref{lien_derivees}, and the stationarity assumption. The right-hand side converges to $0$ as $k$ tends to infinity. Then, up to extraction again, $\frac{\partial}{\partial x_i}u_k(T_x\om)$ converges to $z_{\om,i}$ in $ L^2_{\text{loc}}(\R^d)$ almost surely in $\om$, and therefore,  $v_{\om}$ is in $ H^1_{\text{loc}}(\R^d)$ and $\frac{\partial}{\partial x_i} v_{ \om}=z_{\om,i}$  almost everywhere on $\R^d$.
  \end{proof}

  \begin{prop}\label{propergodic}
  For every $ u$ in $H^1(\Om)$,
      \begin{equation*}
       D_{\om} u \equiv 0 \implies u  \text{ is constant $\mu$-a.s.} 
      \end{equation*}
    The measure $\mu$ is said to be \emph{ergodic} with respect to translations.
  \end{prop}
  \begin{proof}
  Let $u\in H^1(\Om)$ be such that $D_{\om}u\equiv 0$. Then, according to Lemma~\ref{grad}, $\mu$-a.s., a.e. in $ x \in \R^n$,$ \nabla_x u(T_x\om)=0$ so as a function of $L^2_{\text{loc}}(\R^d)$, $ x\mapsto u(T_x\om)$ is constant. As a consequence, $\mu$-a.s., for every $t$ in $\R^d$,
  \begin{equation*}
  u(\om)=\frac{1}{t^d}\int_{[0,t]^d} u(T_x\om)dx  \, ,
  \end{equation*}
  and according to Birkhov theorem, the right-hand side converges $\mu$-a.s. to $\E(u)$ as $t$ goes to infinity, therefore $u$ is constant $ \mu$-a.s.
  \end{proof}

  Now we define spaces that will be used in the definition of the effective matrix for the  homogenized problems in the sequel.
  
  \begin{definition}
     We denote by  $L^2_{\text{pot}}(\Om)$  the closure of the set $\{ D_{\om} u : u \in C^1(\Om)\}$ in $L^2(\Om,\R^d)$, and $L^2_{\text{sol}}(\Om) = L^2_{\text{pot}}(\Om)^\perp$. We also denote by $L^2_{\text{pot,loc}}(\R^d) $ the closure of $\{ \nabla u : u \in C^\oo_c(\R^d)\}$ in $\ L^2_{\text{loc}}(\R^d,\R^d)$, and $L^2_{\text{sol,loc}}(\R^d) $ as the closure of $\{\sigma\in C^\oo_c(\R^d,\R^d):\operatorname{div }\sigma\equiv 0\}$  in $\ L^2_{\text{loc}}(\R^d,\R^d)$.
  \end{definition}
  
  \begin{definition}
      Let $\sigma$ be a random variable in $L^2(\Om,\R^d)$. We say that $\sigma$ is in $H_{\text{div}}(\Om)$ if there exists $g$ in $L^2(\Om)$ such that
      \begin{equation*}
      \forall u \in C^1(\Om),\;  \int_{\Om} \sigma\cdot D_{\om} u\,d\mu(\om) = -\int_{\Om} gu\,d\mu(\om)\,.
      \end{equation*}
      In that case, we denote by $\operatorname{div}_{\om}\sigma$ the function $g$, and call it the divergence of $\sigma$.
  \end{definition}
  \begin{rem}  \label{propdiv}     
   It is clear that $H^1(\Om,\R^d)\subset H_{\text{div}}(\Om)$ and that for $\sigma$ in $H^1(\Om,\R^d)$,
    \begin{equation*}
    \operatorname{div}_{\om} \sigma  = \sum_{i=1}^d D_i \sigma_i\, ,
    \end{equation*}
 where $\sigma_i$, $i=1,..., d$, denote the marginals of $\sigma$.
Moreover, the definition of $L^2_{\text{sol}}(\Om)$ is equivalent to 
\[
L^2_{\text{sol}}(\Om)=\lt\{\sigma\in H_{\text{div}}(\Om) : \operatorname{div}_\om \sigma\equiv 0\rt\}.
\]
  \end{rem}
    
    
   \begin{thm}  \label{l2sol}
  Let $\sigma$ be a random variable in $L^2_{\text{sol}}(\Om)$. Then, $\mu$-almost surely, the function $x\mapsto \sigma(T_x\om)$ is in $L^2_{\text{sol,loc}}(\R^d)$.
  \end{thm}
  
  \begin{proof}
 Using a mollification as in the proof of Lemma~\ref{lemC1denseL2}, we see that $C^1(\Om)$ is dense in $H_{\textrm{div}}(\Om)$ for the natural norm $\|\sigma\|_{H_{\text{div}}}^2:=\|\sigma\|_{L^2(\Om)}^2+\| \operatorname{div}_\om \sigma\|_{L^2(\Om)}^2$. 
 The proof then proceeds as that of  Lemma~\ref{grad}. 
 \end{proof}

  \begin{prop}
  \label{propnondegen}
  The intersection between $L^2_{\text{pot}}(\Om)$ and the space of constant functions on $\Om$ is null.
  The measure $\mu$ is said to be \emph{non-degenerate}.
  \end{prop}
  
  \begin{proof}
 Let  $\xi\in L^2_{\text{pot}}(\Om)$ be a constant function, and $(u_k)_k\in(C^1(\Om))^{\N}$ be a sequence such that $D_{\om}u_k$ converges to $\xi$ in $L^2(\Om)$. 
 Lemma~\ref{lemipp} implies $\E({D_\om u_k})=0$ and therefore,
  \begin{equation*}
  |\xi| = \lt|\E({\xi})\rt| = \lt|\E({D_{\om}u_k - \xi})\rt| \leq\E({ \lt|D_{\om}u_k - \xi\rt| })\leq \lt\|D_{\om}u_k-\xi\rt\|_{L^2(\Om)} \, ,
  \end{equation*}
  according to Cauchy-Schwarz inequality. Sending $k$ to infinity, the right-hand side goes to $0$ which leads to $\xi\equiv 0$.
  \end{proof}

\subsection{Two-scale convergence compactness theorem in $H^1(\mathcal{D})$}
 Having defined $H^1(\Om)$, the purpose of this subsection is to provide a two-scale compactness theorem in $H^1(\mathcal{D})$ in a manner similar to the periodic case (see~\cite{allaire1992homogenization}).

  \begin{lem}
  Let $\{v^{\eps}\}_{\eps\in I}$ be a family of elements of $H^1(\mathcal{D})$ (with $I$ as in Definition~\ref{def}) such that
  \begin{enumerate}[(i)]
  \item $\left\{\|v^{\eps}\|_{L^2(\mathcal{D})}\right\}$ is bounded
  \item $\|\eps\nabla v^{\eps}\|_{L^2(\mathcal{D})} \xrightarrow[\eps\downarrow 0]{} 0$.  
    \end{enumerate}
  Then, up to an extraction, there exists $v^0\in L^2(\mathcal{D})$ such that
  \begin{equation*}
  v^{\eps}(x) \overset{2}{\rightharpoonup} v^0(x) \text{ weakly in $L^2(\mathcal{D}\times \Omega)$}\,.
  \end{equation*}
  \label{notdepend}
  \end{lem}
  
  \begin{proof}
  According to Theorem~\ref{compacite}, there exists $\hat v^0\in L^2(\mathcal{D}\times \Om)$ such that, up to the extraction of a subsequence, $(v^{\eps})_\eps$ two-scale converges weakly to $\hat v^0(x,\om)$. \\
Let $\sigma \in (C^1(\Om))^d$ and $\varphi\in C^1_c(\mathcal{D})$, we have for $0<\eps\leq 1$, 
\[
 \int_{\mathcal{D}} v^{\eps}(x)\varphi (x) \operatorname{div}_\om\sigma (T_{x/\eps}\widetilde{\om})\,dx =- \int_{\mathcal{D}} \eps \nabla_x(v^{\eps}\varphi)(x)\sigma(T_{x/\eps}\widetilde{\om})\,dx \, .
\]
  Passing to the  limit as $\eps$ goes to $0$, and using Cauchy-Schwarz inequality together with assumption (ii), and Proposition \ref{meanvalue}, we infer
  \begin{equation*}
  \int_{\mathcal{D}}\varphi (x) \lt[ \int_{\Om} \hat v^0(x,\om) \operatorname{div}_\om\sigma (\om) \,d\mu(\om)\rt]\,dx = 0 \, .
  \end{equation*}
As a consequence, for almost every $x\in \mathcal{D}$ and every $\sigma\in C^1(\Om)^d$,
\[
\int_{\Om} \hat v^0(x,\om) \operatorname{div}_\om\sigma (\om) \,d\mu(\om)=0.
\]
For such $x$, we have by definition that $\hat v^0(x,\cdot)$ is in $H^1(\Om)$ with $D_\om \hat v^0(x,\cdot)=0$ and by Proposition~\ref{propergodic}, this means that $\hat v^0(x,\om)$ does not depend on $\om$ up to a $\mu$-negligible set. Setting $v^0(x):=\E({\hat v^0(x,\cdot)})$, we have $v^0\in L^2(\mathcal{D})$ and $v^0=\hat v^0$ in $L^2(\mathcal{D}\times \Om)$ so that we still have  the weak two-scale convergence $v^{\eps}(x) \overset{2}{\rightharpoonup} v^0$.
%
  \end{proof}
  
The following compactness theorem is the most important of this section.

  \begin{thm}
  Let  $\{v^{\eps}\}_{\eps\in I}$ be a bounded sequence of $H^1(\mathcal{D})$.  Then, up to an  extraction, there exist $v^0$ in $H^1(\mathcal{D})$ and $\xi$ in $L^2(\mathcal{D},L^2_{\text{pot}}(\Om)) $ such that
  \begin{equation*}
  \left\{
  \begin{aligned}
   & v^{\eps} \overset{2}{\rightharpoonup} v^0(x)  \text{ strongly,} \\
     & \nabla v^{\eps} \overset{2}{\rightharpoonup} \nabla v^0(x) + \xi(x,\om) \text{ weakly, in $L^2(\mathcal{D}\times \Omega)$}\, .
  \end{aligned}\right.
  \end{equation*}
    Moreover, if for every $\eps\in I$, $v^{\eps}$ is in $H_0^1(\mathcal{D})$, then $v^0$ is in $H_0^1(\mathcal{D})$.

\label{coro}
  
  \end{thm}
  \begin{proof}
We first apply Theorem~\ref{compacite} to the family $\{v^\eps\}$ so that for some subsequence $(\eps_k)$,  $(v^{\eps_k})$ two-scale converges towards some $v_0\in L^2(\Omega\times\mathcal{D})$. According to Lemma~\ref{notdepend}, $v^0$ does not depend on $\om$. On the other hand, since $\{v^\eps\}$ is bounded in $H^1(\mathcal{D})$, up to a further extraction, we may assume that $v^\eps$ converges weakly in $H^1(\mathcal{D})$ and strongly in $L^2(\mathcal{D})$  to a limit $\bar{v}\in H^1(\mathcal{D})$. By  Remark~\ref{remstrongconv}, we have actually, $\bar{v}=v^0$ and the sequence strongly two-scale converges to $v^0$.
 
Next, we apply Theorem~\ref{compacite} to the family $\{\nabla v^\eps\}$. Extracting again, the sequence  $(v^{\eps_k})$ two-scale converges towards some element $w$ of $L^2(\Omega\times\mathcal{D},\R^d)$.  Let $\varphi$ in $C_c^{\oo}(\mathcal{D})$. Integrating by parts, we compute,
\[
      \int_{\mathcal{D}} \nabla v^{\eps}(x)\varphi (x) \, dx =  - \int_{\mathcal{D}} v^{\eps}(x)\nabla \varphi(x)\, dx\, ,
\]
  Passing to the two-scale limit, we obtain
  \begin{equation*}
\int_{\mathcal{D}}  \varphi(x)\lt[ \int_{\Om}w(x,\om) \,d\mu(\om) \rt]\,dx\,= - \int_{\mathcal{D}} v^0 \nabla \varphi\, dx\, .
  \end{equation*}
  Therefore, the distribution $\nabla v^0$ is given by the function 
  \[
  \nabla v^0= \int_{\Om}w(\cdot,\om)d\mu(\om)\,,
  \]
   which belongs to $L^2(\mathcal{D})$.

  Now, for $\sigma$ in $L^2_{\text{sol}}(\Om)\cap C^1(\Om,\R^d)$ and  $\varphi$ in $C_c^{\oo}(\mathcal{D})$, we compute
\begin{align*}
      \int_{\mathcal{D}} \nabla v^{\eps}(x)\cdot \sigma(T_{x/\eps}\widetilde\om )\varphi (x) \, dx =&   \int_{\mathcal{D}} \nabla [v^{\eps}(x)\varphi(x)]\cdot \sigma(T_{x/\eps}\widetilde\om )\, dx\ \\
      &\qquad - \int_{\mathcal{D}} v^{\eps}(x)\nabla \varphi(x)\cdot \sigma(T_{x/\eps}\widetilde\om )\, dx\\
      =&  - \int_{\mathcal{D}} v^{\eps}(x)\nabla \varphi(x)\cdot \sigma(T_{x/\eps}\widetilde\om )\, dx\, .
\end{align*}
We have used the fact that by Theorem~\ref{l2sol}, the mapping $x\mapsto \sigma(T_{x/\eps}\widetilde\om)$ is in $L^2_{\text{sol,loc}}(\R^d)$. 
Sending $\eps$ to $0$ and integrating by parts, we get
  \begin{align*}
\int_{\mathcal{D}}      \int_{\Om} w(x,\om)\cdot \sigma(\om)\varphi(x) d\mu \, dx &=- \int_{\mathcal{D}}  \int_{\Om} v^0(x) \nabla  \varphi (x)\cdot \sigma(\om)d\mu(\om) \, dx\\ &=\int_{\mathcal{D}} \int_{\Om}  \nabla  v^0(x)  \cdot  \sigma(\om)\varphi (x)d\mu(\om) \, dx\,.
  \end{align*}
Let us note $\xi:=w-\nabla v^0$.  The last identity shows that for almost every $x$ in $\mathcal{D}$, $\xi(x)$ lies in $((C(\Om))^d\cap L^2_{\text{sol}}(\Om))^{\perp}$. By density of $C(\Om)$ in $L^2(\Om)$ we conclude that 
  \[
 \xi \text{ is in }L^2(\mathcal{D},(L^2_{\text{sol}}(\Om))^{\perp})=L^2(\mathcal{D},L^2_{\text{pot}}(\Om)).
  \]
  This proves the main part of the theorem.

Finally, let us assume moreover, that for every $\eps$, $v^{\eps}$ is in $H_0^1(\mathcal{D})$. Since the subspace $H^1_0(\mathcal{D})$ is closed in $H^1(\mathcal{D})$, we see that $\bar{v}=v^0$ is also in $H_0^1(\mathcal{D})$.
\end{proof}

\section{Applications of stochastic two-scale convergence}
    
     With the tools that we just defined, we are now able to provide homogenization theorems for some elliptic and parabolic problems with random coefficients. To introduce the method and set the notations, we start by considering the well-known problem of diffusion equation. We then turn to non-linear problems, namely harmonic maps and the  Landau-Lifschitz-Gilbert equation which do not have in general a unique solution. From now on we assume $\mathcal{D}$ to be a bounded domain of $\R^3$. 
    
   \subsection{Elliptic equations}
 We consider the problem
\begin{equation}
\left\{
\begin{aligned}
&- \operatorname{div}\left( a(T_{x/\eps}\widetilde{\om})\nabla u^{\eps}(x) \right) = f(x) \, \quad\text{ in }\mathcal{D}\\
&~u^{\eps} \in H_0^1(\mathcal{D}) \, ,
\end{aligned}\right.
\label{pb}
\end{equation}
with  $f$ in $L^2(\mathcal{D})$ and a given $\widetilde{\om}\in \tilde{\Om}$ (see Definition~\ref{def}). The weak formulation of~\eqref{pb} reads 
\begin{equation}
\int_{\mathcal{D}} a(T_{x/\eps}\widetilde{\om})\nabla u^{\eps}(x)\cdot \nabla \varphi(x)dx = \int_{\mathcal{D}} f (x)\varphi(x)dx , \quad \forall\varphi\in C_c^{\oo}(\mathcal{D}) \, ,
\label{eqint}
\end{equation}
which, thanks to Lax-Milgram theorem, is well posed and admits a unique solution.
\begin{definition}\label{defhomogen}
Let us define the homogenized problem as 
\begin{equation}
\left\{
\begin{aligned}
&- \operatorname{div} a^{\text{eff}}\nabla u^0(x)  = f(x) \,, \\
&u^0 \in H_0^1(\mathcal{D}) \, ,
\end{aligned}\right.
\label{pbhomog}
\end{equation}
where $a^{\text{eff}}$ is the $3\times 3$ matrix defined  by
\begin{equation}
a^{\text{eff}}\nu = \int_{\Om} a(\om)(\nu+\eta_\nu(\om))d\mu(\om) \quad \text{for all }\nu\text{ in }\R^3\, ,
\label{aeff}
\end{equation}
with $\eta_\nu$ defined as the unique solution of the problem
\begin{equation}
\left\{
\begin{aligned}
&\eta_\nu\in L^2_{\text{pot}}(\Om)\, , \\
&a(\cdot)(\nu+\eta_\nu(\cdot))  \in L^2_{\text{sol}}(\Om) \, .
\end{aligned}\right. \label{xideff} 
\end{equation}
\end{definition}

\begin{prop}
\label{caraeff}
The matrix $a^{\text{eff}}$ is well defined, symmetric and positive-definite, with the same bounds as $a(x,\om)$, that is $c_1 \Id \leq a^{\text{eff}}\leq c_2 \Id$ (see assumption (H4)). Moreover, it verifies, for every $\nu$ in $\R^3$,
\begin{equation*}
a^{\text{eff}}\nu\cdot\nu = \inf_{v\in L^2_{\text{pot}}(\Omega)} \int_{\Omega} a(\omega)(\nu+v(\omega))\cdot(\nu+v(\omega))d\mu(\omega) \, .
\end{equation*}
\end{prop}
As a consequence, the homogenized problem is well posed and admits a unique solution.
\begin{proof}
The proof is very classical. For every $\nu$ in $\R^3$, let us define
 \begin{eqnarray*}
  Q_{\nu}: L^2_{\text{pot}}(\Om)& \longrightarrow &\R \\
  v &\longmapsto &\int_{\Om}a(\om)(\nu+v(\om))\cdot (\nu+v(\om))d\mu(\om) \, .
    \end{eqnarray*}
  Due to the coercivity assumption on $a$, $Q_{\nu}$ is a strongly convex quadratic functional that possesses a unique minimizer $\xi_{\nu}$ in $L^2_{\text{pot}}(\Om)$. Writing the Euler-Lagrange equation for $\xi_\nu$ leads to~\eqref{xideff} which shows on the one  hand that $\xi_\nu = \eta_\nu$. Using~\eqref{aeff} and~\eqref{xideff} shows on the other hand that
  $$
  a^{\text{eff}}\nu_1\cdot \nu_2 = \int_{\Om}a(\om)(\nu_1+\eta_{\nu_1}(\om))\cdot (\nu_2+\eta_{\nu_2}(\om))d\mu(\om)\, ,
  $$
  which gives that $a^{\text{eff}}$ is symmetric. 
  For the coercivity of $a^{\text{eff}}$, we compute
  \begin{align*}
  a^{\text{eff}}\nu\cdot\nu&= \int_{\Om}a(\om)(\nu+\eta_{\nu}(\om))\cdot (\nu+\eta_{\nu}(\om)) \, d\mu(\om)\geq c_1  \E({ |\nu+\eta_{\nu}(\om)|^2 })\\
  &\geq c_1|\nu|^2 +2 c_1\nu\cdot\E({\eta_\nu})=c_1|\nu|^2\,,
  \end{align*}
 where the identity  $\E({\eta_\nu})=0$ comes from the fact that $\eta_{\nu}$ is in $L^2_{pot}(\Om)$, and Lemma \ref{lemipp}.
  Eventually, we bound $a^{\text{eff}}\nu\cdot\nu$ by using the optimality of $\eta_\nu$. We have
\[
 a^{\text{eff}}\nu\cdot\nu = Q_\nu(\eta_\nu) \, \leq\, Q_\nu(0) =\E({a}) |\nu|^2 \leq c_2 |\nu|^2\,.
\]  
This achieves the proof of the proposition.
\end{proof}

 \begin{rem}
      Let us note that the effective matrix $a^{\text{eff}}$ may not be scalar, even if $a$ is. The following proposition gives a sufficient condition on $a$ so that $a^{\text{eff}}$ is scalar (see~\cite{armstrong2017quantitative} exercise 1.1).
  \end{rem}
     
     \begin{prop}
     Assume that $a$ is scalar valued and \emph{isotropic in law} in the following sense:  for every  linear isometry $\tau$ on $\R^3$, which maps $U=\{\pm e_i: i\in \{1,2,3\}\}$ 
      onto itself, and for every $A$ in $\mathcal{F}$, there holds
     \begin{equation*}
         \mu\{\om\circ\tau: \om\in A\}=\mu(A) \, .
     \end{equation*}
     Then $a^{\text{eff}}$ is scalar.
     \end{prop}

We are now ready to state the homogenization theorem for Problem \eqref{pb}.
    \begin{thm}
    The solution  $u^{\eps}$ of the problem~\eqref{pb}  converges in  $L^2(\mathcal{D})$, and weakly converges in $H^1(\mathcal{D})$ towards the solution  $u^0$ of the homogenized problem~\eqref{pbhomog}.
\end{thm}
    
    \begin{proof}
    Let us first prove that $(u_{\eps})_{\eps}$ is bounded in $H^1(\mathcal{D})$. For every $ \eps>0$, as $a$ is bounded on $\Om$, and $u^{\eps}$ verifies the variational formulation~\eqref{eqint}, it holds that
    \begin{eqnarray*}
         \|\nabla u^{\eps}\|_{L^2(\mathcal{D},\R^3)}^2 &\leq& \frac{1}{c_1} \int_{\mathcal{D}} a(T_{x/ \eps}\widetilde{\om})\nabla u^{\eps}(x)\cdot \nabla u^{\eps}(x)dx\\ &=& \frac{1}{c_1}\int_{\mathcal{D}}f(x) u^{\eps}(x)dx \\ &\leq& \frac{1}{c_1}\|f\|_{L^2(\mathcal{D})}\|u^{\eps}\|_{L^2(\mathcal{D})}\\ &\leq & \frac{C}{c_1}\|f\|_{L^2(\mathcal{D})}\|\nabla u^{\eps}\|_{L^2(\mathcal{D},\R^3)}\, ,
    \end{eqnarray*}
    according to Cauchy-Schwarz and Poincar\'e inequalities.
     Thus, the sequence $(\|\nabla u^{\eps}\|_{(L^2(\mathcal{D}))^3}^2 )_{\eps>0}$ is bounded, and the family of functions $(u^{\eps})_{\eps>0}$ is bounded in $H^1(\mathcal{D})$ (again by Poincar\'e inequality), independently of $\eps$. According to Theorem~\ref{coro} and Rellich Theorem, up to the extraction of a subsequence, we can assume that
    \begin{equation*}
    \left\{
    \begin{aligned}
    & u^{\eps}(x) \overset{2}{\rightharpoonup} u^0(x) \quad\textrm{strongly}\, \text{in } L^2(\mathcal{D})\, ,  \\
      & \nabla u^{\eps}(x) \overset{2}{\rightharpoonup} \nabla u^0(x) + \xi(x,\om)\,  \text{ in } L^2(\mathcal{D}\times \R^3),\\
      & u^{\eps} \to u^0 \text{ weakly in } H^1(\mathcal{D})\, ,
    \end{aligned}\right.
  \end{equation*}
  for some $u^0$ in $H_0^1(\mathcal{D})$ and $\xi$ in  $L^2(\mathcal{D},L^2_{\text{pot}}(\Om))$. 
 Rewriting equation~\eqref{eqint} with the test function $\varphi(x) = \eps \psi(x)v(T_{x/\eps}\widetilde{\om})$, with $\psi\in C_c^{\oo}(\mathcal{D})$ and $v\in C^1(\Om)$, it holds that
  \begin{eqnarray*}
 && \int_{\mathcal{D}} a(T_{x/\eps}\widetilde{\om})\nabla u^{\eps}(x)\cdot \left(\eps v(T_{x/\eps}\widetilde{\om})\nabla \psi (x) + \psi(x)D_{\om}v (T_{x/\eps}\widetilde{\om})\right)dx \\ 
 &&\hspace*{7cm}=  \eps\int_{\mathcal{D}} f(x) \psi(x)v(T_{x/\eps}\widetilde{\om})dx \, .
  \end{eqnarray*}
As $a$ is in $C(\Om,\R^{3\times 3})$, we are allowed to pass to the two-scale limit:
\begin{equation*}
  \int_{\mathcal{D}}\int_{\Om} a(\om)(\nabla  u^0(x)+ \xi(x,\om))\cdot \psi(x)D_{\om} v(\om)d\mu(\om)dx = 0 \, ;
\end{equation*}
hence, a.e. in $\R^3$,
\begin{equation*}
  \int_{\Om} a(\om)(\nabla  u^0(x)+ \xi(x,\om))\cdot D_{\om} v(\om)d\mu(\om) = 0 \, .
\end{equation*}
We thus deduce that $a(\cdot)( \nabla u^0(x)+ \xi(x,\cdot))$ is in $L^2_{\text{sol}}(\Om)$ for a.e. $x\in \mathcal{D}$. As $\xi(x,\cdot)$ is in $L^2_{\text{pot}}(\Om)$, it follows that $\xi(x,\cdot)$ is the unique solution to the problem~\eqref{xideff} with $\nu = \nabla u_0(x)$, thus,
\begin{equation*}
  \int_{\Om} a(\om)(\nabla u^0(x)+\xi(x,\om))\,d\mu(\om)=  a^{\text{eff}}\nabla u^0(x)   \, .
\end{equation*}
Moreover, 
passing to the limit in the two-scale sense in Equation \eqref{eqint}, it holds that for every $\varphi\in C_c^{\oo}(\mathcal{D})$,
\begin{equation*}
  \int_{\mathcal{D}} \int_{\Om}a(\om)(\nabla u^0(x)+\xi(x,\om)) \cdot \nabla\varphi(x)\,d\mu(\om)\,dx  = \int_{\mathcal{D}}f(x)\varphi(x)\,dx\, ,
\end{equation*}
or, in other terms,
\begin{equation*}
  \int_{\mathcal{D}} a^{\text{eff}}\nabla u^0(x)\cdot\nabla\varphi(x) dx = \int_{\mathcal{D}} f(x)\varphi(x)\,dx \, .
\end{equation*}
Therefore, $u^0$ is the unique solution of the homogenized problem~\eqref{pbhomog}. 
\end{proof}
  
\subsection{the harmonic maps equation}
The subject of harmonic maps into manifold is very rich. Giving a complete bibliography on the topic is out
of the scope of this paper and we refer the reader to the review book~\cite{helein2008harmonic} for a rather
complete overview of existing literature on the subject. In terms of physical applications, harmonic maps into the
unit sphere of $\R^3$ have many common features with models that appear when studying liquid crystals or 
ferromagnetic  materials which are the main incentives this study.

Harmonic maps into the unit sphere are critical points of the Dirichlet energy
\begin{equation}
    \mathcal{E}(u) = \frac12 \int_{\mathcal{D}} | \nabla u(x)|^2 \,dx\,,
\label{energy}
\end{equation}
under the constraint $u\in H^1(\mathcal{D},\S^2)$, $\S^2$ being the unit sphere of $\R^3$.
They thus satisfy the harmonic maps equation
\begin{equation*}
    -\Delta u =|\nabla u|^2\,u,
\end{equation*}
in the weak sense, that is to say
\begin{equation}
    \forall \varphi\in H^1_0(\mathcal{D},\R^3),\quad\int_\mathcal{D} \nabla u(x)\cdot \nabla \varphi(x)\,dx = \int_\mathcal{D} |\nabla u(x)|^2 u(x) \cdot  \varphi(x)\,dx \, .
\label{hmeq}
\end{equation}
Existence of non-trivial solutions to~\eqref{hmeq} is guaranteed by the direct method of the calculus of variations, seeking minimizers 
of the energy~\eqref{energy} under the constraint $u\in H^1(\mathcal{D},\S^2)$ and Dirichlet boundary conditions on $\partial \mathcal{D}$. An equivalent weak form was obtained in~\cite{chen1989weak}, namely 
\[
    \forall \varphi\in H^1_0(\mathcal{D},\R^3),\quad     \sum_{i=1}^3 \int_{\mathcal{D}} \left(\partial_i u^{\eps}(x)\times u^{\eps}(x)\right)\cdot \partial_i \varphi(x) \,dx = 0 \, .
 \]
    
In view of applying the methodology described before and in preparation to 
the more complete model of ferromagnetic materials (see section~\ref{secLLG} below), 
we now consider the problem of homogenizing the harmonic maps equation when one
changes the energy~\eqref{energy} to
\[
   \mathcal{E}_\eps(u) = \frac12 \int_{\mathcal{D}}  a(T_{x/\eps}\widetilde{\om})\nabla u^{\eps}(x) \cdot \nabla u^{\eps}(x) \,dx\,,
\]
where again $\widetilde{\om}\in \tilde{\Om}$ is fixed, $a$ is a matrix valued random field satisfying assumptions (H1)--(H5'), $u^\epsilon$ belongs to 
$H^1(\mathcal{D}, S^2)$, and the product $\cdot$ is the scalar product of matrices. Critical points of this energy satisfy the following Euler-Lagrange equation
\begin{equation}
    \left\{
    \begin{aligned}
      &-\operatorname{div}\left(a(T_{x/\eps}\widetilde{\om})\nabla u^{\eps}(x)\right) = \left(a(T_{x/\eps}\widetilde{\om})\nabla u^{\eps}(x) \cdot \nabla u^{\eps}(x)\right) u^{\eps}(x) \,, \\
      &|u^{\eps}(x)| = 1 \; \text{ a.e. on }\mathcal{D} \, .
    \end{aligned}
    \right.
\label{notrepb}
\end{equation}
The above equation must be understood componentwise on $u^\eps$.
Using the same method as in~\cite{chen1989weak}, we have the following proposition whose proof is left to the reader.
\begin{prop}
$u^{\eps} \in H^1(\mathcal{D},\S^2)$ is a weak solution to~\eqref{notrepb} if and only if $u^{\eps}$ satisfies
\begin{equation}
     \sum_{i,j=1}^3 \int_{\mathcal{D}} a_{i,j}(T_{x/\eps}\widetilde{\om})\left( \partial_j u^{\eps}(x)\times u^{\eps}(x)\right)\cdot \partial_i \varphi(x) dx = 0, \quad \forall\varphi \in H_0^1(\mathcal{D},\R^3) \, .
\label{wformeps}
\end{equation}
\end{prop}

Let $a^{\text{eff}}$ be the effective matrix introduced in Definition~\ref{defhomogen} and let us consider the following homogenized problem:
     
\begin{equation}
    \left\{
    \begin{aligned}
        &-\operatorname{div}\left(a^{\text{eff}}\nabla u^0\right)= \left(a^{\text{eff}}\nabla u^0\cdot \nabla u^0\right) u^0 \, ,\\
        &|u^0(x)| = 1 \text{ a.e. on }\mathcal{D} \, .
    \end{aligned}
    \right.
\label{notrepbhomog}
\end{equation}
     
Likewise, weak solutions $u^0 \in H^1(\mathcal{D},\S^2)$ to~\eqref{notrepbhomog} are characterized by
\begin{equation}
     \sum_{i,j=1}^3 a^{\text{eff}}_{i,j}\int_{\mathcal{D}} \left(\partial_j u^0(x)\times u^0(x)\right)\cdot \partial_i \varphi(x)\, dx = 0, \quad \forall\varphi \in H^1_0(\mathcal{D},\R^3) \, .
\label{wfu0}
\end{equation}
     
\begin{thm}
   Let $(u^{\eps})_{\eps>0}$ be a sequence in $H^1(\mathcal{D},\S^2)$ of weak solutions of~\eqref{notrepb}, bounded in $H^1(\mathcal{D},\R^3)$. Then, up to the extraction of a subsequence, $(u^{\eps})_{\eps>0}$ weakly converges in $H^1(\mathcal{D},\R^3)$ to $u^0\in H^1(\mathcal{D},\S^2)$ which is a solution of the homogenized problem~\eqref{notrepbhomog}.
 \label{thmnotrepb}
\end{thm}

\begin{proof}
     By assumption, the sequence $(u^{\eps})_{\eps>0}$ is bounded in $H^1(\mathcal{D},\R^3)$. Therefore, according to Theorem~\ref{coro}, there exist 
     \[
     u^0\in H^1(\mathcal{D},\R^3),\quad \xi\in L^2(\mathcal{D},L^2_{\text{pot}}(\Om,\R^{3\times 3}))
     \]
such that, up to the extraction of subsequences,
\begin{equation*}
    \left\{
    \begin{aligned}
     & u^{\eps}  \overset{2}{\rightharpoonup} u^0 \text{ strongly in } L^2(\mathcal{D},\R^3)\, , \\
     & \forall i \in \{ 1,2,3\},\;\partial_i u^{\eps}  \overset{2}{\rightharpoonup} \partial_i u^0 +\xi_i \, \text{in } L^2(\mathcal{D}\times\Om,\R^3)\, , \\
     & u^{\eps} \rightarrow  u^0 \text{ weakly in } {H^1(\mathcal{D},\R^3)}\, .
    \end{aligned}
    \right.
\end{equation*}
By Rellich Theorem  $(u^{\eps})_{\eps}$ converges to $u^0$ in $L^2(\mathcal{D},\R^3)$ and up to the extraction of a subsequence, we may assume that  $(u^{\eps}(x))_{\eps}$ converges to $u^0(x)$ for a.e. $x\in \mathcal{D}$, and therefore $|u^{0}(x)|=1$ for a.e. $x\in \mathcal{D}$.

Moreover, since for every $\eps>0$ and a.e. on $x$,   $|u^{\eps}(x)|=1$, it holds that, for every $i$ in $\{1,2,3\}$, $\partial_i u^{\eps}(x)\cdot u^{\eps}(x) =\frac12\partial_i |u^\eps(x)|^2 = 0$ almost everywhere on $\mathcal{D}$. Thus, for every $\varphi$ in $C_c^{\oo}(\mathcal{D})$, and $ \psi$ in $C(\Om)$,
\begin{equation*}
   \int_{\mathcal{D}} u^{\eps}(x)\cdot \partial_i u^{\eps}(x) \varphi(x) \psi(T_{x/\eps} \widetilde{\om})dx = 0 \, .
\end{equation*}
According to Proposition~\ref{cv2echforte}, we are allowed to pass to the limit in the above equation to get
\begin{equation*}
    \int_{\mathcal{D}} \int_{\Om}u^0(x)\cdot (\partial_i u^0(x) +\xi_i(x,\om))\varphi(x) \psi(\om)\,d\mu(\om)\,dx= 0 \, .
\end{equation*}
Noticing that $u^0\cdot \partial_i u^0\equiv 0$ a.e. in $\mathcal{D}$, we deduce that, almost everywhere on $\mathcal{D}$ and $\mu$-almost surely,  
\begin{equation}
    u^0(x)\cdot \xi_i(x,\om)=0 \, .
\label{perp2}
\end{equation}

Rewriting the variational formulation \eqref{wformeps} verified by $u^{\eps}$ with the test function $\varphi(x) = \eps \psi(x)v(T_{x/\eps}\widetilde{\om})$, where $\psi\in C_c^{\oo}(\mathcal{D})$ and $v\in C^1(\Om,\R^3)$, it holds that
\begin{eqnarray*}
        &&\eps\sum_{i,j = 1}^3\int_{\mathcal{D}} a_{i,j}(T_{x/\eps}\widetilde{\om}) (\partial_j u^{\eps}(x) \times u^{\eps}(x))\cdot  v(T_{x/\eps}\widetilde{\om})\partial_i \psi(x)dx \\
        &&\hspace*{2cm} +\sum_{i,j=1}^3\int_{\mathcal{D}} a_{i,j}(T_{x/\eps}\widetilde{\om}) (\partial_j u^{\eps}(x) \times u^{\eps}(x))\cdot  D_i v(T_{x/\eps}\widetilde{\om})\psi(x)dx =0\, .
\end{eqnarray*}
The sequence $(u^{\eps})_{\eps}$ strongly two-scale converges to $u^0$, $(\nabla u^{\eps})_{\eps}$ weakly two-scale converges to $\nabla u^0$, and $a$ is in $C(\Om,\R^{3\times 3})$, therefore, applying again Proposition~\ref{cv2echforte}, we are  allowed to pass to the limit and obtain
\begin{equation*}
    \sum_{i,j=1}^3\int_{\mathcal{D}} \int_{\Om}a_{i,j}(\om) ((\partial_j u^0(x)+\xi_j(x,\om)) \times u^0(x))\cdot D_i v(\om) \psi(x)\,d\mu(\om)\,dx =0 \, .
\end{equation*}
As a consequence, almost everywhere in $\mathcal{D}$, and for all $v\in C^1(\Om,\R^3)$,
\begin{equation*}
     \sum_{i,j=1}^3\int_{\Om}a_{i,j}(\om) ((\partial_j u^0(x)+\xi_j(x,\om)) \times u^0(x)) \cdot D_i v(\om)\,d\mu(\om) = 0\, ,
\end{equation*}
or equivalently
\begin{equation}
     \sum_{i,j=1}^3\int_{\Om}a_{i,j}(\om)(\partial_j u^0(x)+\xi_j(x,\om) ) \cdot D_i (v(\om)\times u^0(x))\,d\mu(\om) = 0 \,.
\label{perp}
\end{equation}

For every $v$ in $C^1(\Om,\R^3)$ and $x$ in $\mathcal{D}$, let us write $v$ as  $v =\lambda u^0(x)  + v^{\perp}\times u^0(x) $, with $\lambda = v\cdot u^0(x)\in C^1(\Om)$ and $v^{\perp} = v\times u^0(x) \in C^1(\Om,\R^3)$. Then,
\begin{eqnarray*}
     &&\int_{\Om} a(\om)(\nabla u^0(x) + \xi(x,\om))\cdot D_{\om} v(\om) \,d\mu(\om) \\
     &&\hspace*{1.5cm}= \sum_{i,j=1}^3\int_{\Om}a_{i,j}(\om) (\partial_j u^0(x)+\xi_j(x,\om))  \cdot D_i v(\om) \,d\mu(\om) \\
     &&\hspace*{1.5cm}= \sum_{i,j=1}^3\int_{\Om}a_{i,j}(\om) (\partial_j u^0(x)+\xi_j(x,\om))  \cdot D_i (v^{\perp}(\om)\times u^0(x))\,d\mu(\om)\\ 
     &&\hspace*{2.5cm}+ \sum_{i,j=1}^3\int_{\Om}a_{i,j} (\om)(\partial_j u^0(x)+\xi_j(x,\om))  \cdot u^0(x) D_i \lambda(\om) \,d\mu(\om)\\ 
     &&\hspace*{1.5cm}= 0 \, ,
\end{eqnarray*}
according to equations~\eqref{perp2} and~\eqref{perp}.

It follows that $a(\cdot)(\nabla u^0(x) + \xi(x,\cdot))$ is in $L^2_{\text{sol}}(\Om,\R^{3\times 3})$. As $\xi(x,\cdot)$ is in $L^2_{\text{pot}}(\Om,\R^{3\times 3})$, it is the unique solution to the problem~\eqref{xideff}, with $\nu=\nabla u^0(x)$, and as a consequence,
\begin{equation*}
    \int_{\Om} a(\om)(\nabla u^0(x) + \xi(x,\om))\,d\mu(\om)=a^{\text{eff}}\nabla u^0(x)  \, .
\end{equation*}
Here, the equation must be understood as Equation \eqref{aeff} for each of the coordinate of the vector valued function $u^0$ (with the same matrix $a^{\text{eff}}$).
It is worth noticing that, as in~\cite{alouges2015homogenization}, in that case, the solution $\xi$ to the corrector equation~\eqref{xideff} automatically satisfies~\eqref{perp2}.

Let us now show that $u^0$ verifies the variational formulation of the homogenized problem, i.e. \eqref{wfu0}. For every $\eps>0$ and $\varphi$ in $C_c^{\oo}(\mathcal{D},\R^3)$, $u^{\eps}$ verifies 
\begin{equation*}
     \sum_{i,j=1}^3 \int_{\mathcal{D}} a_{i,j}(T_{x/\eps} \widetilde{\om})(\partial_j u^{\eps}(x)\times u^{\eps}(x))\cdot \partial_i \varphi (x)\,dx = 0 \, .
\end{equation*}
Applying Proposition~\ref{cv2echforte} again, we are allowed to pass to the two limit in the above equation,  and using the Fubini Theorem, we obtain:
\begin{equation*}
     \sum_{i,j=1}^3 \int_{\mathcal{D}} \left(\left(\int_{\Om}a_{i,j}(\om) (\partial_j u^0(x)+ \xi_j(x,\om))\,d\mu(\om))\right)\times u^0(x)\right)\cdot \partial_i \varphi(x) \,dx = 0\,.
\end{equation*}
Moreover,
\begin{align*}
         \MoveEqLeft[3]\sum_{i,j=1}^3 \int_{\mathcal{D}} \left(\left(\int_{\Om}a_{i,j}(\om) (\partial_j u^0(x)+ \xi_j(x,\om))\,d\mu(\om)\right)\times u^0(x)\right)\cdot \partial_i \varphi(x) \,dx \nonumber\\
         &=\sum_{i=1}^3 \int_{\mathcal{D}} \left(\left(\left(a^{\text{eff}}\nabla u^0(x)\right)\cdot e_i\right)\times u^0(x)\right)\cdot \partial_i \varphi(x) \,dx \nonumber \\
         &= \sum_{i,j=1}^3 a^{\text{eff}}_{i,j} \int_{\mathcal{D}}(\partial_j u^0(x)\times u^0(x)) \cdot \partial_i\varphi(x)\,dx \, ,
\end{align*}
and therefore
\begin{equation*}
     \sum_{i,j=1}^3 a^{\text{eff}}_{i,j}\int_{\mathcal{D}} \left(\partial_j u^0(x)\times u^0(x)\right)\cdot \partial_i \varphi(x) dx = 0 \, ,
\end{equation*}
for every $\varphi$ in $C_c^{\infty}(\mathcal{D},\R^d)$, and therefore, by density, for every $\varphi$ in $H_0^1(\mathcal{D},\R^d)$, which is nothing but~\eqref{wfu0}. Therefore, $u^0$ is a weak solution to~\eqref{notrepbhomog}.
\end{proof}
     
\subsection{the Landau-Lifshitz-Gilbert equation}
\label{secLLG}

We now turn to the homogenization of Landau-Lifshitz-Gilbert equation. Let us set $Q_T =(0,T)\times \mathcal{D}$, with $T>0$, and let us consider, for every $\eps>0$, the problem
\begin{equation}
     \left\{ 
     \begin{aligned}
     \frac{\partial u^{\eps}}{\partial t}& =  u^{\eps}\times \operatorname{div}(a(T_{x/\eps}\widetilde{\om})\nabla u^{\eps})  - \lambda u^{\eps} \times \left(u^{\eps}\times (\operatorname{div}(a(T_{x/\eps}\widetilde{\om})\nabla u^{\eps})\right)\, , \\
     & u^{\eps}(0,x) = u_0(x) \, ,
     \end{aligned}
     \right.
\label{landaulif}
\end{equation}
with $\tilde{\omega}\in \tilde{\Omega}$ fixed according to  Definition~\ref{def} and $u_0$ a function of $H^1(\mathcal{D},\S^2)$. We make use of the definition of weak solutions from~\cite{alouges1992global}.
     
\begin{definition}
Let $u^{\epsilon}$ be a function defined on $Q_T$. We say that $u^{\epsilon}$ is a weak solution of~\eqref{landaulif} if it verifies:
\begin{enumerate}[(D1)]
     \item the function $u^{\eps}$ is in $H^1(Q_T,\S^2)$;
     \item for every $\phi$ in $C^\oo_c(Q_T,\R^3)$,
          \begin{multline*}
        \int_{Q_T}\left( \frac{\partial u^{\eps}}{\partial t}(t,x)  + \lambda  u^{\eps}(t,x) \times \frac{\partial u^{\eps}}{\partial t}(t,x)\right)\cdot \phi(t,x)\,dx\,dt   \\= (1+\lambda^2) \sum_{i,j=1}^3 \int_{Q_T} a_{i,j}(T_{x/\eps}\widetilde{\om})\left( \partial_j u^{\eps}(t,x)\times u^{\eps}(t,x)\right)\cdot\partial_i \phi (t,x)\,dx\,dt  \, ;
    \end{multline*}
    \item $u^{\eps}(0,\cdot) = u_0$ in the sense of traces;
    \item the Dirichlet energy is uniformly bounded independently of $t$: for every $t$ in $(0,T)$,
             \begin{multline*}
     \mathcal{E}_{\eps}(u^{\eps})(t):=\\
     \frac{1}{2}\int_{\mathcal{D}} a(T_{x/\eps}\widetilde{\om})\nabla u^{\eps}(t,x)\cdot \nabla u^{\eps}(t,x)\,dx +\frac{\lambda}{1+\lambda^2} \,\int_0^t\int_{\mathcal{D}}\left|\frac{\partial u^{\eps}}{\partial t}\right|^2 \,dx\,ds\\ 
     \leq c_2 \int_{\mathcal{D}}  |\nabla u_0|^2(x) \,dx\, .
         \end{multline*}
         \end{enumerate}    
\label{defwsol}
\end{definition}
     
Existence of weak solutions of~\eqref{landaulif} has been established in~\cite{alouges1992global,visintin1985landau}, using a Galerkin approximation.
We also consider the homogenized problem
\begin{equation}
     \left\{ 
     \begin{aligned}
    & \frac{\partial \bar{u}}{\partial t}  =  \bar{u}\times \operatorname{div}(a^{\text{eff}}\nabla \bar{u}) - \lambda \bar{u} \times \left(\bar{u}\times (\operatorname{div}(a^{\text{eff}}\nabla \bar{u})\right) \, ,\\
     & \bar{u}(x,0) = u_0(x) \, ,
     \end{aligned}
     \right.
     \label{landaulifhomog}
\end{equation}
with $a^{\text{eff}}$ the effective matrix introduced in Definition~\ref{defhomogen}. We define a weak solution $\bar u$ of \eqref{landaulifhomog} as in Definition~\ref{defwsol}, 
replacing the diffusion matrix $a(T_{x/\eps}\widetilde{\om})$ by $a^{\text{eff}}$.

In order to apply the methodology of Section 2 and 3.1 in this more complicated case, we need a proper notion of two-scale convergence for time dependent problems that do not 
present fast oscillations in time. We thus introduce the following definition.

\begin{definition}
We say  that $v^{\eps}(t,x)$ weakly  two-scale converges to $v(t,x,\om)$ in $L^2(Q_T)$ if $(v^{\eps})_{\eps>0}$ is a bounded sequence of  $L^2(Q_T)$ and for every $ \phi$ in $C_c^{\oo}(Q_T), $ and $ b$ in $C(\Om)$,
\begin{equation*}
  \lim_{\eps\to 0}  \int_{Q_T} v^{\eps}(t,x)\phi (t,x) b (T_{x/\eps}\widetilde{\om})\,dx\,dt = \int_{Q_T} \int_{\Om} v(t,x,\om)\phi (t,x) b(\om) d \mu (\om) \,dx\,dt \, .
\end{equation*}
\end{definition}
     
\begin{rem}
The definition of two-scale convergence in $L^2(Q_T)$ corresponds to Definition~\ref{def}, with $T_{(t,x)}=T_{x}$. Defined as such, $T$ is  an ergodic action that does not depend on $t$.  In other words, for every $u$ in $L^2(\Om)$, $\om$ in $\Om$, and $t$ in $(0,T)$, $u(T_{(t,0)}\om)=u(\om)$, and therefore $D_t u=0$. It follows that every $z$ in $L^2_{\operatorname{pot}}(\Om)$ has a vanishing first coordinate, and we thus consider $L^2_{\operatorname{pot}}(\Om)$ as naturally embedded into  $L^2(\Om,\R^3)$. The results of Section 2 thus still hold in this case, among which Theorem~\ref{coro}, stated in the following proposition with the needed modifications.
\end{rem}     

\begin{prop}\label{propH1}
Let $(v^{\eps})_{\eps>0}$ be a bounded sequence of   $H^1(Q_T)$. Then, there exist $v^0$ in $H^1(Q_T)$ and $\xi$ in $L^2(Q_T,L^2_{\text{pot}}(\Om, \R^{3\times 3}))$ such that, up to the extraction of a subsequence,
\[       \left\{
       \begin{aligned}
       & v^{\eps} {\longrightarrow} v^0 \text{ weakly in }H^1(Q_T)\,,& \\
       & \forall i \in \{1,2,3\},\;  \partial_i v^{\eps}\overset{2}{\rightharpoonup} \partial_i v^0 + \xi_i\, \text{in } L^2(Q_T\times \Omega)\, , & \\
       & \frac{\partial v^{\eps}}{\partial t} \overset{2}{\rightharpoonup} \frac{\partial v^0}{\partial t} \, \text{in } L^2(Q_T\times \Om)\, .& 
       \end{aligned}\right.
\]
\end{prop}

The main result of the paper is the following.

\begin{thm}
Let $(u^{\eps})_{\eps>0}$ be a family of weak solutions of~\eqref{landaulif} such that the energy is uniformly bounded independently of $\eps$ (i.e. the constant $K$ in the Definition~\ref{defwsol} does not depend on $\eps$). Then, up to the extraction of a subsequence, $u^{\eps}$ weakly converges in $H^1(Q_T,\R^3)$ as $\eps$ goes to 0  to $\bar u$ which is a weak solution of the homogenized problem~\eqref{landaulifhomog}.
\end{thm}
\begin{proof}
  Let us first demonstrate that the sequence $(u^{ \eps})_{\eps}$ is bounded in $H^1(Q_T,\R^3)$. 
       As $Q_T$ is a bounded domain and, for every $\eps>0$, $u^{\eps}$ is in $H^ 1(Q_T,\S^ 2)$, it follows that $(u^{\eps})_{\eps>0}$ is bounded in the norm of $L^2(Q_T,\R^3)$.
For every $\eps>0$, using the uniform ellipticity of $a$,
\begin{eqnarray*}
    \|\nabla u^{\eps}\|_{L^2(Q_T,\R^{3\times 3})}^2 
         &=& \int_{Q_T}|\nabla u^{\eps}|^2\,dx\,dt + \int_{Q_T} \left|\frac{\partial u^{\eps}}{\partial t}\right|^2 \,dx\,dt\\
         &\leq&  \frac{1}{c_1}\int_0^T \int_{\mathcal{D}} a(T_{x/\eps}\widetilde{\om})\nabla u^{\eps}(t,x)\cdot \nabla u^{\eps}(t,x)\,dx\,dt\\&&+c_2 \int_{\mathcal{D}}  |\nabla u_0|^2(x) \,dx\\
         &\leq&  \left(\frac{2T}{c_1}+1\right)c_2 \int_{\mathcal{D}}  |\nabla u_0|^2(x) \,dx\, .
\end{eqnarray*}

     As a consequence, according to Proposition~\ref{propH1} and Rellich Theorem, there exist $\bar{u}$ in $H^1(Q_T,\R^3)$, $\xi$ in $L^2(Q_T,L^2_{\text{pot}}(\Om,\R^3))$ such that, up to the extraction of a subsequence,
\begin{equation*}
     \left\{
     \begin{aligned}
     & u^{\eps} \overset{2}{\rightharpoonup} \bar{u} \, \text{ in } L^2(Q_T\times \Omega, \R^3)\, , \\
     & \forall i \in \{ 1,2,3\},\;\partial_i u^{\eps} \overset{2}{\rightharpoonup} \partial_i \bar{u} + \xi_i \,\text{ in } L^2(Q_T\times \Omega, \R^3)\, , \\
     & \frac{\partial u^{\eps}}{\partial t} \overset{2}{\rightharpoonup} \frac{\partial \bar{u}}{\partial t} \, \text{ in } L^2(Q_T\times \Omega,\R^3)\, , \\
    & u^{\eps}\rightarrow \bar{u} \, \text{ in } {L^2(Q_T,\R^3)} \, ,\\
     & u^{\eps} \rightarrow  \bar{u} \text{ weakly in } {H^1(Q_T,\R^3)}\, .
     \end{aligned}\right.
\end{equation*}
Moreover, we may also assume that $u^\eps$ converges to $\bar{u}$ a.e. in $Q_T$ and,
since $|u^\eps(t,x)|=1$ for a.e. $(t,x)\in Q_T$, we deduce that $\bar{u}\in H^1(Q_T,\S^2)$. We have established that $\bar{u}$ verifies (D1).
     
For every $\eps>0$, $u^{\eps}(0,x)=u_0(x)$ in the trace sense, hence almost everywhere on $\mathcal{D}$ . By continuity of the trace in $L^2(\mathcal{D})$ with respect to weak convergence in $H^1(Q_T)$, we have
\begin{equation*}
     \bar{u}(0,\cdot) = \lim_{\eps\to 0} u^{\eps}(0,\cdot) = u_0\quad\text{in }L^2(\mathcal{D})\,.
\end{equation*}
Therefore $\bar{u}$ complies to (D3).

     Testing  the variational formulation (D2) against $\phi(t,x) = \eps \psi(t,x)v(T_{x/\eps}\widetilde{\om})$, where $\psi\in C_c^{\oo}(Q_T)$ and $v\in C^1(\Om,\R^3)$, it holds that
     \begin{multline*}
       \eps \int_{Q_T} \left(\frac{\partial u^{\eps}}{\partial t}(t,x) + \lambda   u^{\eps}(t,x) \times \frac{\partial u^{\eps}}{\partial t}(t,x)\right)\cdot  v(T_{x/\eps}\widetilde{\om})\psi(t,x) \,dx\,dt  \\ 
    =(1+\lambda^2)\left( \eps\sum_{i,j=1}^3\int_{Q_T} a_{i,j}(T_{x/\eps}\widetilde{\om}) (\partial_j u^{\eps} \times u^{\eps})\cdot  v(T_{x/\eps}\widetilde{\om})\partial_i \psi(t,x)\,dx\,dt\right. \\ 
 \left. +\sum_{i,j=1}^3\int_{Q_T} a_{i,j}(T_{x/\eps}\widetilde{\om}) (\partial_j u^{\eps} \times u^{\eps})\cdot  D_i v(T_{x/\eps}\widetilde{\om})\psi(t,x)\,dx\,dt\right) .
         \end{multline*}

      We recall that the sequence $(u^{\eps})_{\eps>0}$ strongly two-scale converges to $\bar{u}$ in $L^2$, and for every $j$ in $\{1,2,3\}$ the sequences $(\partial_j u^{\eps})_{\eps>0}$ and $(\frac{\partial u^{\eps}}{\partial t})_{\eps>0}$ weakly two-scale converge to  $\partial_j\bar{u} +\xi_j$ and $\frac{\partial \bar{u}}{\partial t}$ respectively. According to Proposition~\ref{cv2echforte}, and  as $a$ is in $C(\Om,\R^{3\times 3})$,  we are allowed to pass to the limit in the above equation, which yields\
\begin{align*}
\sum_{i,j=1}^3\int_{Q_T} \int_{\Om}a_{i,j}(\om) ((\partial_j \bar{u}(t,x)+\xi_j(t,x,\om)) \times \bar{u}(t,x))\cdot D_i v(\om) \,\psi(t,x) \,d\mu(\om) \,dx\,dt&\\ \MoveEqLeft[5]=0 \, .
\end{align*}
It follows that, almost everywhere in $Q_T$, and for any $v \in C^1(\Omega, R^3)$,
\begin{equation*}
     \sum_{i,j=1}^3\int_{\Om}a_{i,j}(\om) ((\partial_j \bar{u}(t,x)+\xi_j(t,x,\om)) \times \bar{u}(t,x) )\cdot D_i v(\om)\,d\mu(\om) = 0\, ,
\end{equation*}
that we rewrite as 
\begin{equation}
     \sum_{i,j=1}^3\int_{\Om}a_{i,j}(\om) (\partial_j\bar{u}(t,x)+\xi_j(t,x,\om))  \cdot D_i (v\times \bar{u}(t,x))\,d\mu(\om) = 0 \,,\label{perp0}
\end{equation}
by invariance of the triple product under circular permutation.   \\

Moreover, almost everywhere in $Q_T$, $|u^{\eps}(t,x)|=1$, therefore, for every $i$ in $\{1,2,3\}$, $\varphi$ in $C_c^{\oo}(Q_T)$, and $ \psi$ in $C^1(\Om)$, 
\begin{equation*}
     \int_{Q_T} u^{\eps}(t,x)\cdot \partial_i u^{\eps}(t,x) \varphi(x) \psi(T_{x/\eps} \widetilde{\om})\,dx\,dt = 0 \, .
\end{equation*}
      As the sequence $(u^{\epsilon})_{\epsilon>0}$ strongly two-scale converges to $\bar{u}$ and $(\partial_i u^{\epsilon})_{\epsilon}$ weakly two-scale converges to $\partial_i \bar{u}+\xi_i$, applying Proposition~\ref{cv2echforte} again, we can pass to the limit in the above equation to get
\begin{equation*}
 \int_{Q_T} \int_{\Om}\bar{u}(t,x)\cdot (\partial_i \bar{u}(t,x) +\xi_i(t,x,\om))\varphi(t,x) \psi(\om)d\mu(\om)dxdt  = 0 \, .
\end{equation*}
Consequently, almost everywhere in $Q_T$ and $\mu$-almost surely,
\begin{equation}
     \bar{u}(t,x)\cdot (\partial_i \bar{u}(t,x)+\xi_i(t,x,\om))=0 \, .\label{perp02}
\end{equation}
Decomposing the test function $v$ in~\eqref{perp0} along and orthogonally to $\bar{u}$, and using the same arguments as in the proof of Theorem~\ref{thmnotrepb}, we deduce from~\eqref{perp0},~\eqref{perp02} that $a(\cdot)(\nabla \bar{u}(t,x) + \xi(t,x,\cdot))$ is  in 
$L^2_{\text{sol}}(\Om,\R^{3\times3})$. Thus, as $\xi(t,x,\cdot)$ is in $L^2_{\text{pot}}(\Om,\R^{3\times3})$, $\xi(t,x,\cdot)$ is the unique solution to the problem~\eqref{xideff}, with $\nu=\nabla \bar{u}(t,x)$, and hence
     
\begin{equation}
       \int_{\Om} a(\om)(\nabla \bar{u}(t,x) + \xi(t,x,\om))\,d\mu(\om)=a^{\text{eff}}\nabla \bar{u}(t,x)\, .
\label{aeffLLG}       
\end{equation}

    Let us now show that $\bar{u}$ verifies the variational formulation  of the homogenized problem~\eqref{landaulifhomog}. For every $\eps>0$ and $\phi$ in $C^\oo_c(Q_T,\R^3)$, $u^{\eps}$ being a weak solution of~\eqref{landaulif} satisfies 
\begin{multline*}
        \int_{Q_T}\left( \frac{\partial u^{\eps}}{\partial t}(t,x)  + \lambda  u^{\eps}(t,x) \times \frac{\partial u^{\eps}}{\partial t}(t,x)\right)\cdot \phi(t,x)\,dx\,dt   \\= (1+\lambda^2) \sum_{i,j=1}^3 \int_{Q_T} a_{i,j}(T_{x/\eps}\widetilde{\om}) \left(\partial_j u^{\eps}(t,x)\times u^{\eps}(t,x)\right)\cdot\partial_i \phi (t,x)\,dt\,dx  \, .
\end{multline*}
     As seen above, for every $j$ in $\{1,2,3\}$, $\partial_j u^{\eps}\times u^{\eps}$,$\frac{\partial u^{\eps}}{\partial t}$   and  $u^{\eps}\times\frac{\partial u^{\eps}}{\partial t}$ weakly two-scale converge in $L^2(Q_T\times \Om, \R^3)$, respectively to $(\partial_j \bar{u}+\xi_j)\times \bar{u}$,$\frac{\partial \bar{u}}{\partial t}$ and $\bar{u}\times\frac{\partial \bar{u}}{\partial t}$. The matrix $a$ being in $C(\Om,\R^{3\times 3})$, we may pass to the two-scale limit in the above equation, which yields
\begin{eqnarray*}
      &&\int_{Q_T} \left(\frac{\partial \bar{u}}{\partial t} + \lambda \bar{u}\times \frac{\partial \bar{u}}{\partial t}\right)\cdot \phi \,dt\,dx \\
      &&\hspace*{2cm}=(1+\lambda^2) \sum_{i,j=1}^3 \int_{Q_T} \left(\left(\int_{\Om}a_{i,j} (\partial_j \bar{u}+ \xi_j)\,d\mu\right)\times \bar{u}\right)\cdot \partial_i \phi \,dt\,dx\\
      &&\hspace*{2cm}=(1+\lambda^2)    \sum_{i,j=1}^3 a^{\text{eff}}_{i,j}\int_{Q_T}  \left(\partial_j\bar{u}\times \bar{u}\right)\cdot\partial_i \phi  \, dt\,dx,
\end{eqnarray*}
from~\eqref{aeffLLG}. We conclude that $\bar{u}$ solves the variational formulation of the homogenized problem (D2), with $a_{i,j}$ replaced by $a^\text{eff}_{ij}$.\\

In order to finish the proof, it remains to show that $\bar{u}$ satisfies the energy dissipation inequality (D4), again with $a$ replaced by $a^\text{eff}$.
Let $t$ in $(0,T)$ be fixed. As $u^{\eps}$ weakly converges to $\bar{u}$ in $H^1(Q_T,\R^3)$, by continuity of the trace, it follows that $u^{\eps}(t)$ weakly converges to $\bar{u}(t)$ in $L^2(\mathcal{D},\R^3)$. Besides, as $u^{\eps}$ verifies the inequality (D4) of dissipation of the energy, and since $|u^\epsilon(t,x)|=1$ for a.e. $(t,x) \in Q_T$ , 
it holds that, for every $\eps>0$,
\begin{align*}
    \|u^{\eps}(t,\cdot)\|_{H^1(\mathcal{D},\R^3)}^2 & = \int_{\mathcal{D}} |u^{\eps}(t,x)|^2\,dx + \int_{\mathcal{D}} |\nabla u^{\eps}(t,x)|^2\,dx \\
         &\leq |\mathcal{D}| + \frac{2c_2}{c_1} \int_{\mathcal{D}}  |\nabla u_0|^2(x) \,dx\, 
\end{align*}
by the uniform ellipticity of $a$. Therefore, according to Theorem~\ref{coro}, there exist $v$ in $H^1(\mathcal{D},\R^3)$ and $v_1$ in $L^2(\mathcal{D}, L^2_{\text{pot}}(\Om,\R^{3\times 3}))$ such that
\begin{equation*}
    \left\{
        \begin{aligned}
        u^{\eps}(t) &{\longrightarrow} v \text{ weakly in } H^1(\mathcal{D})\, ,\\
        \nabla u^{\eps}(t) &\overset{2}{\rightharpoonup} \nabla v + v_1 \, \text{ in } L^2(\mathcal{D}\times \Om)\, .
        \end{aligned}
    \right.
\end{equation*}
By uniqueness of the weak limit in $L^2(\mathcal{D},\R^3)$, $v=\bar{u}(t,\cdot)$. Let $(\tilde{v}^n_1)_{n\in \mathbb{N}}$ be a sequence of elements in the linear span of test functions in $C_c^{\oo}(\mathcal{D})\times C(\Om,\R ^{3\times 3})$ such that $\tilde{v}^n_1 \rightarrow v_1$ in $L^2(\mathcal{D}\times\Om,\R^{3\times 3})$ as $n$ tends to infinity. For every $\eps>0$, $t$ in $(0,T)$, by the positivity of $a$,
\begin{align*}
\int_{\mathcal{D}} a(T_{x/\eps}\widetilde{\om})\left(\nabla u^{\eps}(t,x)-\nabla \bar{u}(t,x)- \tilde{v}^n_1(x,T_{x/\eps}\widetilde{\om})\right)&\\ \MoveEqLeft[7]\cdot\left( \nabla u^{\eps}(t,x)-\nabla \bar{u}(t,x) -\tilde{v}^n_1(x,T_{x/\eps}\widetilde{\om})\right)\,dx\geq 0\,.
\end{align*}
Expanding this equation gives
\begin{align*}
\int_{\mathcal{D}} a(T_{x/\eps}\widetilde{\om})\nabla u^{\eps}(t,x)\cdot \nabla u^{\eps}(t,x)\,dx& \\\MoveEqLeft[13]\geq 2\int_{\mathcal{D}} a(T_{x/\eps}\widetilde{\om})\nabla u^{\eps}(t,x)\cdot(\nabla \bar{u}(t,x)+\tilde{v}^n_1(x,T_{x/\eps}\widetilde{\om}))\,dx \\
 \MoveEqLeft[12]-\int_{\mathcal{D}} a(T_{x/\eps}\widetilde{\om})(\nabla \bar{u}(t,x) +\tilde{v}^n_1(x,T_{x/\eps}\widetilde{\om}))\cdot (\nabla \bar{u}(t,x) +\tilde{v}^n_1(x,T_{x/\eps}\widetilde{\om}))\,dx \, .
\end{align*}
    For every $n$ in $\N$, $\widetilde{v}_1^{n}$ is in the linear span of test functions, therefore we are allowed to pass to the two-scale limit in the right-hand side term. Thus, as we pass to the inferior limit in the above inequality as $\eps$ goes to zero, it holds that, 
\begin{align*}
\liminf_{\eps\downarrow0}\int_{\mathcal{D}} a(T_{x/\eps}\widetilde{\om})\nabla u^{\eps}(t,x)\cdot \nabla u^{\eps}(t,x)\,dx& \\\MoveEqLeft[15]\geq 2\int_{\mathcal{D}}\int_{\Om}a(\om)(\nabla \bar{u}(t)+v_1(x,\om))\cdot(\nabla\bar{u}(t) +\tilde{v}^n_1(x,\om))\,d\mu(\om) \,dx\\
\MoveEqLeft[13]
 -\int_{\mathcal{D}}\int_{\Om} a(\om)(\nabla \bar{u}(t) +\tilde{v}^n_1(x,\om))\cdot (\nabla \bar{u}(t) +\tilde{v}^n_1(x,\om))\,d\mu(\om) \,dx\, .
\end{align*}
Passing now to the limit as $n$ tends to infinity, we obtain
\begin{align*}
\liminf_{\eps\downarrow0}\int_{\mathcal{D}} a(T_{x/\eps}\widetilde{\om})\nabla u^{\eps}(t,x)\cdot \nabla u^{\eps}(t,x)\,dx &\\
\MoveEqLeft[15]\geq \int_{\mathcal{D}} \int_{\Om}a(\om)(\nabla \bar{u}(t,x)+v_1(x,\om))\cdot(\nabla\bar{u}(t,x)+v_1(x,\om))\,d\mu(\om)\,dx
&\\
\MoveEqLeft[15]\geq \int_{\mathcal{D}}a^{\text{eff}}\nabla \bar{u}(t)\cdot\nabla \bar{u}(t) \,dx \, .
\end{align*}
As a consequence, for every $t$ in $(0,T)$, as $(u^{\eps})_{\eps>0}$ weakly converges to $\bar{u}$ in $H^1((0,t)\times\mathcal{D},\R^3)$, we are allowed to pass to the limit in (D4),
using the lower semi-continuity of the weak limit in $L^2(Q_t)$:
\begin{align*}
\frac{1}{2}\int_{\mathcal{D}}a^{\text{eff}}\nabla \bar{u}(t,x)\cdot\nabla \bar{u}(t,x)\,dx +\frac{\lambda}{1+\lambda^2} \int_0^t\int_{\mathcal{D}}\left|\frac{\partial \bar{u}}{\partial t}\right|^2(s,x)\,dx\,ds &\\\MoveEqLeft[5]\leq c_2 \int_{\mathcal{D}}  |\nabla u_0|^2(x) \,dx\, .
\end{align*}
Finally, $\bar{u}$ verifying (D1)--(D4) is a weak solution of~\eqref{landaulifhomog}.\end{proof}

\begin{rem}
In~\cite{alouges1992global}, the existence  of weak solutions of~\eqref{landaulif} has been proven with an even more precise bound on the dissipation of the energy given by the energy at time 0. Namely,
\begin{equation*}
     \mathcal{E}_\eps = \int_{\mathcal{D}} a(T_{x/\eps}\widetilde{\om}) \nabla u_0(x)\cdot \nabla u_0(x) \,dx .
\end{equation*}
Notice that, on the one hand, $(\mathcal{E}_\eps)_{\eps>0}$ is uniformly bounded, since by the uniform boundedness of $a$, one has
\begin{equation*}
     \mathcal{E}_\eps \leq c_2 \int_{\mathcal{D}}  |\nabla u_0|^2(x) \,dx\, .
\end{equation*}
On the other hand, $(\mathcal{E}_\eps)_{\epsilon}$ converges due to Birkhov ergodic theorem
\begin{equation*}
    \lim_{\eps\downarrow 0} \mathcal{E}_\eps =\int_{\mathcal{D}}\E{(a)}\nabla u_0(x)\cdot  \nabla u_0(x)\,dx \, ,
\end{equation*}
which is \emph{not} the homogenized energy at time 0. Indeed,
\begin{align*}
\int_{\mathcal{D}} a^{\text{eff}}\nabla u_0(x)\cdot\nabla u_0(x)\,dx &\\
\MoveEqLeft[10]= \int_{\mathcal{D}}\inf_{v\in L^2_{\text{pot}}(\Om,\R^3)} \int_{\Om}a(\om)(\nabla u_0(x) + v(\om))\cdot(\nabla u_0(x) +v(\om))\,d\mu(\om)\,dx\\
\MoveEqLeft[10]\leq\int_{\mathcal{D}}\E({a}) \nabla u_0(x)\cdot \nabla u_0(x)\,dx\,,
\end{align*}
the last inequality being strict in general.
\end{rem}

\section*{Conclusion}
In this paper, we have presented a simple version of the stochastic two-scale convergence that was proposed in~\cite{zhikov2006homogenization}. The technique provides tools that are very similar to the ones used in periodic homogenization, and mostly
extends the techniques to the stochastic setting. We have proposed two application examples of the method in non-linear problems, one concerning the homogenization of the harmonic maps equation while the other deals with the homogenization of the Landau-Lifshitz-Gilbert equation, which is currently used for the modelization of ferromagnetic materials that are inhomogeneous at the small scale.

We emphasize the fact that both equations share common mathematical difficulties such as the non-linear behavior, a non convex constraint on the (vectorial) unknown and non uniqueness of the weak solutions. Nevertheless, we have shown that the stochastic two-scale convergence method may be applied to those equations, permitting to identify the limiting homogenized equation.

We are of course aware of the fact that, for ferromagnetic materials, the full model contains several other terms that need to be considered (only the so-called exchange term has been considered so far) and an extension to more complete models is currently under study. With the frame we defined in this article, we have all the tools we need to handle these generalizations.

The next step of this approach is to make numerical experiments. Yet, we notice that the expression of the effective matrix $a^{\text{eff}}$ is not explicit: we need to compute it numerically. The approximation of $a^{\text{eff}}$ have been extensively studied in the recent years, in particular by Gloria and Otto (see for instance\cite{gloria2015corrector}) and Armstrong, Kuusi and Mourrat (see \cite{armstrong2017quantitative}). Combined with suitable numerical experiments, we believe that the approach might be able to fully identify and study relevant models for composite ferromagnetic materials such as spring magnets or nanocristalline ferromagnets.

\end{document}